\documentclass[a4paper]{amsart}                

\usepackage[latin1]{inputenc}                  
\usepackage[T1]{fontenc}                       
\usepackage[spanish,english]{babel}            
\usepackage{amsmath,amssymb,amsthm}            
\usepackage{latexsym}                          
\usepackage{delarray}                          
\usepackage{bbm}                               
\usepackage{hyperref}                          
\usepackage[pdftex,usenames,dvipsnames]{color} 
\usepackage{bbding,trfsigns}                   
\usepackage{wasysym}                           
\usepackage{datetime}                          

\DeclareMathAlphabet{\mathpzc}{OT1}{pzc}{m}{it} 

\newtheorem{Th}{Theorem}[section]              

\newtheorem{Lem}{Lemma}[section]

\newcommand{\G}{\Gamma}
\DeclareMathOperator{\supp}{supp}

\title[UMD Banach spaces, BMO functions and Bessel convolutions]
      {Characterization of Banach valued BMO functions and UMD Banach spaces by using Bessel convolutions}

\author[J.J. Betancor]{Jorge J. Betancor}
\author[A.J. Castro]{Alejandro J. Castro}
\author[L. Rodríguez-Mesa]{Lourdes Rodríguez-Mesa}

\address{\newline
        Jorge J. Betancor, Alejandro J. Castro, and Lourdes Rodríguez-Mesa \newline
        Departamento de Análisis Matemático,
        Universidad de la Laguna, \newline
        Campus de Anchieta, Avda. Astrofísico Francisco Sánchez, s/n, \newline
        38271, La Laguna (Sta. Cruz de Tenerife), Spain}

\email{jbetanco@ull.es, ajcastro@ull.es, lrguez@ull.es}

\keywords{$\gamma$-radonifying operators, BMO, Bessel convolution, UMD Banach spaces}

\subjclass[2010]{46E40, 42A50}

\thanks{The authors are partially supported by MTM2010/17974. The second author is also supported by a FPU grant from the Government of Spain}

\begin{document}


  \maketitle                

  \begin{abstract}
    In this paper we consider the space $BMO_o(\mathbb{R},X)$ of bounded mean oscillations and odd functions on $\mathbb{R}$ taking values in a UMD Banach
    space $X$. The functions in $BMO_o(\mathbb{R},X)$ are characterized by Carleson type conditions involving Bessel convolutions and $\gamma$-radonifying norms.
    Also we prove that the UMD Banach spaces are the unique Banach spaces for which certain $\gamma$-radonifying Carleson inequalities for Bessel-Poisson integrals
    of $BMO_o(\mathbb{R},X)$ functions hold.
  \end{abstract}

    \section{Introduction} \label{sec:intro}

    As it is well known the Hilbert transform $H$ defined by
    $$Hf(x)= \text{P.V.} \int_{\mathbb{R}} \frac{f(y)}{x-y} dy, \quad \text{a.e. } x \in \mathbb{R},$$
    is bounded from $L^p(\mathbb{R})$ into itself, for every $1<p<\infty$, and from $L^1(\mathbb{R})$ into $L^{1,\infty}(\mathbb{R})$. If $X$ is a Banach space,
    the Hilbert transform is defined on $L^p(\mathbb{R}) \otimes X$, $1 \leq p < \infty$, in the natural way. A Banach space $X$ is said to be a UMD space when the
    Hilbert transform can be extended to the Bochner-Lebesgue space $L^p(\mathbb{R},X)$ as a bounded operator from $L^p(\mathbb{R},X)$ into itself, for some
    (equivalently, for any) $1<p<\infty$. Equivalent definitions and other properties and applications of UMD Banach spaces can be found in \cite{Amann},
    \cite{Bou}, \cite{Bu1}, \cite{Bu2}, \cite{Bu3}, \cite{GiWe}, \cite{Hy2}, \cite{HW}, \cite{KaWe} and \cite{Rub},  amongst others.\\

    In \cite{BCFR1} and \cite{BCFR2} it was studied the space $BMO_o(\mathbb{R})$ of odd bounded mean oscillation on $\mathbb{R}$. $BMO_o(\mathbb{R})$ was characterized
    by using Carleson measures involving Poisson and heat integrals associated with Bessel operators (\cite[Theorem 1.1]{BCFR2}). In this paper we consider the Banach valued
    odd $BMO$ space. Assume that $X$ is a Banach space. We say that a function $f \in L^1_{loc}(\mathbb{R},X)$ belongs to $BMO(\mathbb{R},X)$, when
    $$\|f\|_{BMO(\mathbb{R},X)} = \sup_{I \subset \mathbb{R}} \frac{1}{|I|} \int_I \|f(x)-f_I\|_X dx<\infty,$$
    where the supremum is taken over all bounded intervals $I \subset \mathbb{R}$. Here $f_I = \frac{1}{|I|} \int_I f(x)dx$,  the integral being
    understood in the Bochner sense, and $|I|$ denotes the length of $I$. By $BMO_o(\mathbb{R},X)$
    we represent the space of all odd functions
    in $BMO(\mathbb{R},X)$. According to the John-Nirenberg inequality we can see that an odd function
    $f \in L^1_{loc}(\mathbb{R},X)$ is in $BMO_o(\mathbb{R},X)$ if, and only if, for some (equivalently, for any)
    $1 \leq p < \infty$, there exists $C>0$ such that
    \begin{equation} \label{1.1}
        \left( \frac{1}{|I|} \int_I \|f(x)-f_I\|_X^p dx \right)^{1/p}
            \leq C,
    \end{equation}
    for every interval $I=(a,b)$, $0<a<b<\infty$, and
    \begin{equation}\label{1.2}
        \left( \frac{1}{|I|} \int_I \|f(x)\|_X^p dx \right)^{1/p}
            \leq C,
    \end{equation}
    for each interval $I=(0,b)$, $0<b<\infty$. Moreover, for every $f \in BMO_o(\mathbb{R},X)$ and
    $1\leq p < \infty$, $\|f\|_{BMO(\mathbb{R},X)}$ is equivalent to the infimum of the constants $C$ satisfying  \eqref{1.1} and \eqref{1.2}. \\

    For every $\lambda>0$, we consider the Bessel operator $\Delta_{\lambda} = -x^{-\lambda} \frac{d}{dx}x^{2\lambda}\frac{d}{dx}x^{-\lambda}$ on $(0,\infty)$.
    If $J_\nu$ denotes the Bessel function of the first kind and order $\nu$, we have that
    \begin{equation}\label{2.1}
        \Delta_{\lambda,x}\left( \sqrt{xy} J_{\lambda-1/2}(xy) \right) =y^2  \sqrt{xy} J_{\lambda-1/2}(xy), \quad x,y \in (0,\infty).
    \end{equation}
    The Hankel transformation $h_\lambda$ is defined by
    $$h_\lambda(f)(x)=\int_0^\infty \sqrt{xy} J_{\lambda-1/2}(xy) f(y) dy, \quad x \in (0,\infty),$$
    for every $f \in L^1(0,\infty)$. The transformation $h_\lambda$ plays in the Bessel setting the same role as the Fourier transformation in the classical
    (Laplacian) setting. $h_\lambda$ is an isometry in $L^2(0,\infty)$ and $h_\lambda^{-1}=h_\lambda$ on $L^2(0,\infty)$.\\

    We represent by $S_\lambda(0,\infty)$ the space constituted by all functions $\phi \in C^\infty(0,\infty)$ such that, for every $m,k \in \mathbb{N}$,
    $$\beta_{m,k}^\lambda(\phi)=\sup_{x \in (0,\infty)} \left| x^m \left( \frac{1}{x} \frac{d}{dx}\right)^k \left( x^{-\lambda} \phi(x)\right) \right| < \infty. $$
    $S_\lambda(0,\infty)$ is endowed with the topology generated by the family of seminorms $\{\beta_{m,k}^\lambda\}_{m,k \in \mathbb{N}}$. The Hankel transformation
    $h_\lambda$ is an automorphism in $S_\lambda(0,\infty)$ (\cite[Lemma 8]{Ze}).
    The dual space of $S_\lambda(0,\infty)$ is denoted by $S_\lambda(0,\infty)'$.\\

    If $f,g \in L^1((0,\infty),x^\lambda dx)$ the Bessel (also called Hankel) convolution $\#_\lambda$ is defined by
    $$(f \#_\lambda g)(x)= \int_0^\infty f(y) \,_\lambda\tau_x (g)(y) dy, \quad x \in (0,\infty),$$
    where the Bessel translation $\,_\lambda\tau_x (g)$ of $g$ is given by
    $$\,_\lambda\tau_x (g)(y)
        = \frac{(xy)^\lambda}{\sqrt{\pi} 2^{\lambda-1/2}\G(\lambda)} \int_0^\pi (\sin \theta)^{2\lambda-1}
          \frac{g \left( \sqrt{(x-y)^2+2xy(1-\cos \theta)} \right)}{\left((x-y)^2+2xy(1-\cos \theta)\right)^{\lambda/2}}  d\theta,
      \quad x,y \in (0,\infty),$$
    (see \cite{GS}). By \cite[Theorem 2.d]{Hi} we have the following interchange formula,
    \begin{equation}\label{hank_conv}
        h_\lambda(f \#_\lambda g)(x) = x^{-\lambda} h_\lambda(f)(x) h_\lambda( g)(x), \quad x \in (0,\infty).
    \end{equation}
    From \eqref{hank_conv} it is clear that $\#_\lambda$ is a commutative and associative operation in
    $L^1((0,\infty),x^\lambda dx)$.\\

    The $\#_\lambda$-convolution was studied on $S_\lambda(0,\infty)$ and $S_\lambda(0,\infty)'$ in \cite{MB}.
    The mapping $(\phi, \psi) \longmapsto \phi \#_\lambda \psi $ is bilinear and continuous from
    $S_\lambda(0,\infty) \times S_\lambda(0,\infty)$ into $S_\lambda(0,\infty)$ (\cite[Proposition 2.2, $(i)$]{MB}).\\

    The Hankel convolution $f \#_\lambda \phi$ is defined when $f \in L^1((0,\infty),x^\lambda dx; X)$ and
    $\phi \in L^1((0,\infty),x^\lambda dx)$ in the natural way, that is, understanding the integrals in the Bochner's sense.\\

    If $\phi: \Upsilon \longrightarrow \mathbb{R}$, where $\Upsilon=\mathbb{R}$ or $\Upsilon=(0,\infty)$, we denote by
    $\phi_{(t)}$ and $\phi_t$, $t>0$, the following dilated functions
    $$\phi_{(t)}(x)=\phi_{(t)}^\lambda(x)=\frac{1}{t^{\lambda+1}} \phi\left(\frac{x}{t}\right), \quad
      \phi_{t}(x)=\frac{1}{t} \phi\left(\frac{x}{t}\right) , \quad t \in (0,\infty), \ x \in \Upsilon.$$
    While $\phi_t$ is the classical dilated function, $\phi_{(t)}$ is the one adapted to the $\Delta_\lambda$-setting.\\

    By $\gamma\left(L^2((0,\infty)^2,\frac{dydt}{t^2});X\right)$ we represent the Gauss space, also called
    $\gamma$-radonifying operators from $L^2((0,\infty)^2,\frac{dydt}{t^2})$ into $X$
    (see \cite{HW} and \cite{Ne} for general definitions and properties). Suppose that $F : (0,\infty)^2 \longrightarrow X$
    is weakly-$L^2((0,\infty)^2,\frac{dydt}{t^2};X)$, that is, $G=\langle F, x'\rangle \in L^2((0,\infty)^2,\frac{dydt}{t^2})$,
    for every $x' \in X'$, where $X'$ denotes the dual space of $X$. We say that
    $F \in \gamma\left(L^2((0,\infty)^2,\frac{dydt}{t^2});X\right)$ when the operator $I_F$ defined by
    $$I_F(h)=\int_{(0,\infty)^2} F(y,t)h(y,t) \frac{dtdy}{t^2}, \quad h \in L^2\left((0,\infty)^2,\frac{dydt}{t^2}\right),$$
    belongs to $\gamma\left(L^2((0,\infty)^2,\frac{dydt}{t^2});X\right)$, that is,
    \begin{align*}
    \|I_F\|&_{\gamma\left(L^2((0,\infty)^2,\frac{dydt}{t^2});X\right)}
        =  \sup \left\| \sum_j \gamma_j
            I_F(h_j) \right\|_{L^2(\Omega,X)} <\infty,
   \end{align*}
   where the supremum is taken over all finite orthonormal families $\{h_j\}$ in $L^2((0,\infty)^2,\frac{dydt}{t^2})$
   and $\{\gamma_j\}_{j=1}^\infty$ is a sequence of independent complex standard Gaussian random variables on
   some probability space $(\Omega, \mathcal{A},\mathbb{P})$. In this case, we write
   $\|F\|_{\gamma\left(L^2((0,\infty)^2,\frac{dydt}{t^2});X\right)}$ to refer
   $\|I_F\|_{\gamma\left(L^2((0,\infty)^2,\frac{dydt}{t^2});X\right)}$. If $F$ is not weakly-$L^2((0,\infty)^2,\frac{dydt}{t^2};X)$
   we say that $\|F\|_{\gamma\left(L^2((0,\infty)^2,\frac{dydt}{t^2});X\right)}=\infty$.\\

   Hytönen and Weis \cite{HW} introduced vector valued versions of some functionals considered by Coifman, Meyer
   and Stein \cite{CMS} to define tent spaces. Here we work with truncated
   versions of Hytönen and Weis' functionals. For every $x,r \in (0,\infty)$ we define the truncated
   cones
   $$\G_+(x)=\{(y,t) \in (0,\infty)^2 : |x-y|<t\},$$
   and
   $$\G_+^r(x)=\{(y,t) \in (0,\infty)^2 : |x-y|<t<r\}.$$
   Let $F : (0,\infty)^2 \longrightarrow X$ be a strongly measurable function. We define the conical square function
   $A^+(F)$ as follows
   $$A^+(F)(x)
    =\|F(y,t) \chi_{\G_+(x)}(y,t)\|_{\gamma\left(L^2((0,\infty)^2,\frac{dydt}{t^2});X\right)}, \quad x \in (0,\infty).$$
    Also, we consider for every $r>0$ the truncated square function $A^+(F\big|r)$ given by
    $$A^+(F \big|r)(x)
        =\|F(y,t) \chi_{\G_+^{r}(x)}(y,t)\|_{\gamma\left(L^2((0,\infty)^2,\frac{dydt}{t^2});X\right)}, \quad x \in (0,\infty).$$
    Note that, according to \cite[p. 499]{HW}, if $X=\mathbb{C}$, then
    $$A^+(F)(x)
        = \left( \int_{\G_+(x)} |F(y,t)|^2 \frac{dydt}{t^2} \right)^{1/2}, \quad x \in (0,\infty).$$
    For every $0<q<\infty$, we define $C_q^+(F)$ by
    $$C_q^+(F)(z)
        =\sup_{I \owns z} \left( \frac{1}{|I|} \int_I A^+(F \big||I|/2)^q(x) dx \right)^{1/q}, \quad z \in (0,\infty),$$
    where the supremum is taken over all bounded intervals $I \subset (0,\infty)$ such that $z \in I$.
    $C_q^+(F)$ is somehow a $q$-average of the truncated conical square function.\\

    We now state our first result that can be seen as a Bessel version of \cite[Theorem 1.1]{HW}.

    \begin{Th}\label{Th_principal}
        Let $X$ be a UMD Banach space and  $\lambda>1$. Assume that $f$ is an odd $X$-valued function satisfying
        that $(1+x^2)^{-1}f \in L^1(\mathbb{R},X)$ and $\phi \in S_\lambda(0,\infty)$
        such that $\int_0^\infty x^\lambda \phi(x)dx=0$.
        \begin{itemize}
            \item[$(i)$] If $f \in BMO_o(\mathbb{R},X)$, then $C_q^+(f \#_\lambda \phi_{(t)}) \in L^\infty(0,\infty)$,
                for every $0<q<\infty$.
            \item[$(ii)$] If $C_q^+(f \#_\lambda \phi_{(t)}) \in L^\infty(0,\infty)$, for some $0<q<\infty$,
                then $f \in BMO_o(\mathbb{R},X)$.
        \end{itemize}
        Moreover, the quantities $\|f\|_{BMO_o(\mathbb{R},X)}$ and $\|C_q^+(f \#_\lambda \phi_{(t)})\|_{L^\infty(0,\infty)}$,
        $0<q<\infty$, are equivalent.
    \end{Th}

    The Poisson semigroup $\{P_t^\lambda\}_{t>0}$ associated with the Bessel operator $\Delta_\lambda$
    (that is, the semigroup of operator generated by $-\sqrt{\Delta_\lambda}$) is given by
    $$P_t^\lambda(f)(x)=\int_0^\infty P_t^\lambda(x,y)f(y)dy, \quad t,x \in (0,\infty),$$
    where the Poisson kernel $P_t^\lambda(x,y)$, $ t,x,y \in (0,\infty),$ is defined by (\cite{Weinst})
    \begin{equation*}\label{Bessel_kernel}
        P_t^\lambda(x,y)
            =\frac{2\lambda (xy)^\lambda t}{\pi}
            \int_0^\pi \frac{(\sin \theta)^{2\lambda-1}}{[(x-y)^2+t^2+2xy(1-\cos \theta)]^{\lambda+1}} d\theta, \quad t,x,y \in (0,\infty).
    \end{equation*}
    The semigroup $\{P_t^\lambda\}_{t>0}$ is contractive in $L^p(0,\infty)$, $1 \leq p \leq \infty$, but
    it is not Markovian. \\

    In \cite{BFMT} Littlewood-Paley $g$-functions associated with $\{P_t^\lambda\}_{t>0}$ acting on Banach valued functions were defined. If $1<q<\infty$ and
    $f : (0,\infty) \longrightarrow X$ is a strongly measurable function the $g_q^\lambda$-function of $f$ is defined
    by
    $$g_q^\lambda(f)(x)
        = \left( \int_0^\infty \| t \partial_t P_t^\lambda(f)(x) \|_X^q \frac{dt}{t} \right)^{1/q}, \quad x \in (0,\infty).$$
    Those Banach spaces that admit a $q$-uniformly convex or $q$-uniformly smooth (see \cite{Pi} for definitions)
    equivalent norms were characterized by using $L^p$-inequalities involving $g_q^\lambda$-functions
    (\cite[Theorems 2.4 and 2.5]{BFMT}).\\

    Also we can define the $q$-conical square function $G_q^\lambda(f)$ of $f$ associated to the Poisson semigroup
    $\{P_t^\lambda\}_{t>0}$ by
    $$G_q^\lambda(f)(x)
        = \left( \int_{\G_+(x)} \| t \partial_t P_t^\lambda(f)(y) \|_X^q \frac{dy dt }{t^2} \right)^{1/q}, \quad x \in (0,\infty).$$
    By taking into account the results in \cite{MTX} and by using the ideas developed
    in the proof of \cite[Theorems 2.4 and 2.5]{BFMT} (see also \cite[Proposition 1.3]{BCR}) the Banach spaces having $q$-uniformly
    convex or smooth renorming can be characterized in the following way.

    \begin{Th}\label{Th_conical_Lp}
        Let $X$ be a Banach space and $\lambda>0$.
        \begin{itemize}
            \item[$(i)$]  Suppose that $2 \leq q < \infty$. The space $X$ admits a $q$-uniformly convex equivalent norm if,
            and only if, for some (equivalently, for any) $1<p<\infty$,
            $$\|G_q^\lambda(f)\|_{L^p(0,\infty)}
                \leq C \|f\|_{L^p((0,\infty),X)}, \quad f \in L^p((0,\infty),X).$$
            \item[$(ii)$] Assume that $1<q \leq 2$. The space $X$ can be $q$-uniformly smooth renormed if, and only if,
            for some (equivalently, for any) $1<p<\infty$,
            $$\|f\|_{L^p((0,\infty),X)}
                \leq C \|G_q^\lambda(f)\|_{L^p(0,\infty)} , \quad f \in L^p((0,\infty),X).$$
        \end{itemize}
    \end{Th}

    Also in \cite{BCR} the authors characterize the Banach spaces with a $q$-uniformly convex
    and smooth equivalent norm by using Carleson measures and the space $BMO_o(\mathbb{R},X)$.

    \begin{Th}[{\cite[Theorems 1.1 and 1.2]{BCR}}]\label{Th_BCR}
        Let $X$ be a Banach space and $\lambda>1$.
        \begin{itemize}
            \item[$(i)$] Assume that $2 \leq q < \infty$. Then, $X$ has an equivalent norm which is $q$-uniformly convex
            if, and only if, there exists  $C>0$ such that, for every $f \in BMO_o(\mathbb{R},X)$,
            $$  \sup_{I} \frac{1}{|I|} \int_0^{|I|} \int_I \|t \partial_t P_t^\lambda(f)(y) \|^q_X \frac{dy dt}{t}
                    \leq C \|f\|^q_{BMO_o(\mathbb{R},X)},$$
            where the supremum is taken over all bounded intervals $I$ in $(0,\infty)$.
            \item[$(ii)$] Suppose that $1<q \leq 2$. Then, $X$ has an equivalent $q$-uniformly smooth norm
            if, and only if, there exists  $C>0$ such that, for every odd
             $X$-valued function $f$, satisfying that $(1+x^2)^{-1}f \in L^1(\mathbb{R},X)$,
            $$  \|f\|^q_{BMO_o(\mathbb{R},X)}
                    \leq C  \sup_{I} \frac{1}{|I|}\int_0^{|I|} \int_I \|t \partial_t P_t^\lambda(f)(y) \|^q_X \frac{dy dt}{t},$$
            where the supremum is taken over all bounded intervals $I$ in $(0,\infty)$.
        \end{itemize}
    \end{Th}

    Note that
    \begin{align*}
        \sup_{I} \frac{1}{|I|}\int_0^{|I|} \int_I \|t \partial_t P_t^\lambda(f)(y) \|^q_X \frac{dy dt}{t}
            \backsimeq  \sup_{I} \frac{1}{|I|} \int_I \int_{\G_+^{|I|/2}(x)} \|t \partial_t P_t^\lambda(f)(y) \|^q_X \frac{dy dt}{t^2} dx.
    \end{align*}

    For every $t>0$, the operator $P_t^\lambda$ is a Hankel convolution operator. Indeed, if we define
    $$k^\lambda(x)= \frac{2^{\lambda+1/2} \G(\lambda+1)}{\sqrt{\pi}} \frac{x^\lambda}{(1+x^2)^{\lambda+1}}, \quad x \in (0,\infty),$$
    then $P_t^\lambda(f)=f \#_\lambda k_{(t)}^\lambda$, $f \in L^p(0,\infty)$, $1 \leq p < \infty$, and $t>0$.
    Moreover, we can write
    $$t \partial_t P_t^\lambda(f)
        = f \#_\lambda t \partial_t k_{(t)}^\lambda, \quad f \in L^p(0,\infty), \ t>0.$$
    Note that
    \begin{align*}
        t \partial_t k_{(t)}^\lambda(x)
            = \frac{1}{t^{\lambda+1}} \left[ -(\lambda+1) k^\lambda \left( \frac{x}{t} \right)
                - \frac{x}{t} \left( \frac{d}{du} k^\lambda \right)(u)_{\Big|u=x/t} \right]
            = \mathpzc{h}_{(t)}^\lambda(x), \quad t,x \in (0,\infty),
    \end{align*}
    where
    $$\mathpzc{h}^\lambda(x)
        = -(\lambda+1) k^\lambda(x)
                - x \frac{d}{dx} k^\lambda(x), \quad x \in (0,\infty).$$
    It is not hard to see that $\int_0^\infty x^\lambda \mathpzc{h}^\lambda(x) dx=0$ and that $\mathpzc{h}^\lambda \notin S_\lambda(0,\infty)$.
    Our next result cannot be deduced from Theorem~\ref{Th_principal} because $\mathpzc{h}^\lambda \notin S_\lambda(0,\infty)$,
    but $\mathpzc{h}^\lambda$ has sufficient decay so the computations given in the proof of
    Theorem~\ref{Th_principal} (see Section~\ref{sec:proof_Th}) remain valid, even for $\lambda>0$ (see \cite[Proposition 4.4]{BCFR2}).

    \begin{Th}\label{Th_Poisson}
        Let $X$ be a UMD Banach space, $\lambda>0$ and $q>0$. Then, there exists $C>0$ such that
        $$\frac{1}{C}\|f\|_{BMO_o(\mathbb{R},X)} \leq \|C_q^+(t \partial_t P_t^\lambda(f)) \|_{L^\infty(0,\infty)} \leq C \|f\|_{BMO_o(\mathbb{R},X)},$$
        for every $f \in BMO_o(\mathbb{R},X)$.
    \end{Th}

    Note that if $X=\mathbb{C}$, we have that
    \begin{align*}
        \|C_2^+(t \partial_t P_t^\lambda(f)) \|_{L^\infty(0,\infty)}^2
            = & \sup_{I} \frac{1}{|I|} \int_I \int_{\G_+^{|I|/2}(x)} |t \partial_t P_t^\lambda(f)(y) |^2 \frac{dy dt}{t^2} dx \\
            \backsimeq & \sup_{I} \frac{1}{|I|}\int_0^{|I|} \int_I |t \partial_t P_t^\lambda(f)(y) |^2 \frac{dy dt}{t}.
    \end{align*}
    For a general Banach space $X$, even when $q=2$,
    $$\|C_q^+(t \partial_t P_t^\lambda(f)) \|_{L^\infty(0,\infty)}^q
        \not \backsimeq  \sup_{I} \frac{1}{|I|}\int_0^{|I|} \int_I \|t \partial_t P_t^\lambda(f)(y) \|^q_X \frac{dy dt}{t},
        \quad f \in BMO_o(\mathbb{R},X).$$
    Indeed, if $X$ is a UMD Banach space and
    $$\|C_2^+(t \partial_t P_t^\lambda(f)) \|_{L^\infty(0,\infty)}^2
        \backsimeq  \sup_{I} \frac{1}{|I|}\int_0^{|I|} \int_I \|t \partial_t P_t^\lambda(f)(y) \|^2_X \frac{dy dt}{t},
        \quad f \in BMO_o(\mathbb{R},X),$$
    Theorems~\ref{Th_BCR} and \ref{Th_Poisson} implies that $X$ has an equivalent $2$-uniformly convex norm
    and an equivalent $2$-uniformly smooth norm, hence $X$ is isomorphic to a Hilbert space (see \cite[Proposition 3.1]{Kw}
    and \cite[Proposition 4.36]{Pi}), and this is not always possible. For example, $L^p(\mathbb{R})$, $1<p<\infty$, $p \neq 2$, is a UMD space which is not
    isomorphic to a Hilbert space. \\

    We now establish a new characterization of the UMD Banach spaces in terms of the Poisson semigroup
    $\{P_t^\lambda\}_{t>0}$ and the functional $C_q^+$.

    \begin{Th}\label{Th_caract}
        Let $X$ be a Banach space, $\lambda>1$ and $0<q<\infty$. Then the following assertions are equivalent.
        \begin{itemize}
            \item[$(i)$] $X$ is UMD.
            \item[$(ii)$] There exist $C>0$ such that, for every odd $X$-valued function satisfying that $(1+x^2)^{-1}f \in L^1(\mathbb{R},X)$,
            \begin{equation}\label{tdt}
                \frac{1}{C}\|f\|_{BMO_o(\mathbb{R},X)} \leq \|C_q^+(t \partial_t P_t^\lambda(f)) \|_{L^\infty(0,\infty)},
            \end{equation}
            and
            \begin{equation}\label{tdx}
                \left\|C_q^+\left(\int_0^\infty f(z) t D_{\lambda,z} P_t^\lambda(x,z) dz\right)\right\|_{L^\infty(0,\infty)} \leq C \|f\|_{BMO_o(\mathbb{R},X)},
            \end{equation}
            where $D_{\lambda,z}=z^\lambda \frac{d}{dz} z^{-\lambda}$.
        \end{itemize}
    \end{Th}

    This paper is organized as follows. In Section~\ref{sec:lem} we establish some auxiliary results. We prove
    $L^p$-boundedness properties of $A^+$ and a polarization identity involving
    $\#_\lambda$-convolution. The proof of Theorem~\ref{Th_principal} is presented in Section~\ref{sec:proof_Th}.
    Finally in Section~\ref{sec:Poisson} we give a proof of Theorem~\ref{Th_caract}.\\

    Throughout this paper we denote by $C$ a positive constant that can change in each occurrence.

    \section{Auxiliary results}  \label{sec:lem}

    In this section we establish some results that will be useful in the proof of Theorem~\ref{Th_principal}.
    Firstly we prove two boundedness properties of $A^+$.

    \begin{Lem}\label{Lem_Lp}
        Let $X$ be a UMD Banach space, $\lambda>0$ and $1<p<\infty$. Assume that $\phi \in S_\lambda(0,\infty)$,
        verifying that $x^\lambda \phi$ has vanishing integral over $(0,\infty)$. Then, there exists $C>0$ such that
        $$\|A^+(f \#_\lambda \phi_{(t)})\|_{L^p(0,\infty)}
            \leq C \|f\|_{L^p((0,\infty),X)}, \quad f \in L^p((0,\infty),X).$$
    \end{Lem}

    \begin{proof}
        Let $f \in L^p((0,\infty),X)$. To simplify the notation we call $\psi(w)=\phi(w)w^{-\lambda}$, $w \in (0,\infty)$.
        According to \cite[p. 85]{EG}, there exists $\Phi$ in the Schwartz class $S(\mathbb{R})$ such that $\Phi(w^2)=\psi(w)$, $w \in (0,\infty)$.
        We also introduce the function,
        \begin{equation*}\label{Psi}
            \Psi(w)=\frac{1}{\sqrt{\pi}2^{\lambda+1/2}\G(\lambda)} \int_0^\infty u^{\lambda-1} \Phi(w^2+u)du, \quad w \in \mathbb{R}.
        \end{equation*}
        We can write
        \begin{align*}
            \,_\lambda\tau_y (\phi_{(t)})(z)
            = & \frac{(yz)^\lambda}{\sqrt{\pi} 2^{\lambda-1/2}\G(\lambda)t^{2\lambda+1}}
            \left(\int_0^{\pi/2}  +  \int_{\pi/2}^\pi \right) (\sin \theta)^{2\lambda-1} \Phi \left( \frac{(y-z)^2+2yz(1-\cos \theta)}{t^2} \right) d\theta\\
            = & \mathcal{K}_1(z;y,t) + \mathcal{K}_2(z;y,t), \quad t,y,z \in (0,\infty).
        \end{align*}

        Taking into account that $\Phi \in S(\mathbb{R})$, we obtain
        \begin{align}\label{K2}
            \|\mathcal{K}_2&(z;y,t)\|_{L^2(\G_+(x),\frac{dydt}{t^2})}
                \leq   C z^\lambda \left\{ \int_{\G_+(x)} \frac{y^{2\lambda}}{t^{4\lambda+4}} \left( \frac{t^2}{t^2 + y^2+z^2} \right)^{2\lambda+2} dydt \right\}^{1/2} \nonumber \\
                \leq &  C z^\lambda \left\{ \int_0^\infty \int_{|x-y|}^\infty  \frac{ dt dy}{(t+y+z)^{2\lambda+4}} \right\}^{1/2}
                = C z^\lambda \left\{ \int_0^\infty  \frac{dy}{(|x-y|+y+z)^{2\lambda+3}}  \right\}^{1/2} \nonumber \\
                = & C z^\lambda \left\{ \int_0^x  \frac{dy}{(x+z)^{2\lambda+3}} + \int_x^\infty  \frac{dy}{(2y-x+z)^{2\lambda+3}}  \right\}^{1/2}
                \leq C \frac{z^\lambda}{(x+z)^{\lambda+1}} \leq  \frac{C}{x+z}, \ x,z \in (0,\infty).
        \end{align}
        On the other hand, by proceeding in a similar way we have that
        $$\|\mathcal{K}_1(z;y,t)\|_{L^2(\G_+(x),\frac{dydt}{t^2})}
            \leq C z^\lambda \left\{ \int_0^\infty  \frac{y^{2\lambda} }{(|x-y|+|y-z|)^{4\lambda+3}}  dy\right\}^{1/2} , \ x,z \in (0,\infty).$$
        In order to analyze this integral we consider two situations. Firstly, assume that $0 < z \leq x/2$. Then,
        \begin{align*}
            \int_0^\infty  \frac{y^{2\lambda}}{(|x-y|+|y-z|)^{4\lambda+3}} dy
                \leq & C \Big( z^{2\lambda} \int_0^z  \frac{dy}{(x+z-2y)^{4\lambda+3}}
                  + x^{2\lambda} \int_z^x  \frac{dy}{(x-z)^{4\lambda+3}}
                  + \int_{x}^\infty \frac{dy}{(2y-x-z)^{2\lambda+3}}  \Big)\\
                \leq & C \left( \frac{x^{2\lambda}}{(x-z)^{4\lambda+2}} + \frac{1}{(x-z)^{2\lambda+2}} \right)
                \leq \frac{C}{(x-z)^{2\lambda+2}}.
        \end{align*}
        By symmetry reasons, we also have that
        $$\int_0^\infty  \frac{y^{2\lambda}}{(|x-y|+|y-z|)^{4\lambda+3}} dy
            \leq \frac{C}{(z-x)^{2\lambda+2}}, \quad 0<2x \leq z<\infty.$$
        Thus,
        \begin{equation}\label{K1_global}
            \|\mathcal{K}_1(z;y,t)\|_{L^2(\G_+(x),\frac{dydt}{t^2})}
                \leq C \frac{z^\lambda}{|x-z|^{\lambda+1}}
                \leq C \left\{ \begin{array}{rl}
                        \dfrac{1}{x}, & 0 < z \leq x/2,\\
                        & \\
                        \dfrac{1}{z}, & 0 < 2x \leq z.\\
                     \end{array} \right.
        \end{equation}

        We now introduce the new kernels
        $$\mathcal{K}_{1,1}(z;y,t)
            = \frac{(yz)^\lambda}{\sqrt{\pi} 2^{\lambda-1/2}\G(\lambda)t^{2\lambda+1}}
                \int_0^{\pi/2}  \theta^{2\lambda-1} \Phi \left( \frac{(y-z)^2+2yz(1-\cos \theta)}{t^2} \right) d\theta,$$
        and
        $$\mathcal{K}_{1,2}(z;y,t)
            = \frac{(yz)^\lambda}{\sqrt{\pi} 2^{\lambda-1/2}\G(\lambda)t^{2\lambda+1}}
                \int_0^{\pi/2}  \theta^{2\lambda-1} \Phi \left( \frac{(y-z)^2+yz\theta^2}{t^2} \right) d\theta.$$
        By using the mean value theorem, the decay of $\Phi$ and that $2(1-\cos \theta) \sim \theta^2$, when $\theta \in (0,\pi/2)$, we can write,
        \begin{align}\label{K1-K11}
            \|\mathcal{K}_{1}&(z;y,t) - \mathcal{K}_{1,1}(z;y,t)\|_{L^2(\G_+(x),\frac{dydt}{t^2})} \nonumber \\
                \leq &  C  \left\{ \int_{\G_+(x)}
                \left( \frac{(yz)^\lambda}{t^{2\lambda+2}} \int_0^{\pi/2} |(\sin \theta)^{2\lambda-1}-\theta^{2\lambda-1}|
                \left| \Phi \left( \frac{(y-z)^2+2yz(1-\cos \theta)}{t^2} \right) \right| d\theta\right)^2 dydt\right\}^{1/2} \nonumber \\
                \leq &  C  \left\{ \int_{\G_+(x)}
                \left( \frac{(yz)^\lambda}{t^{2\lambda+2}} \int_0^{\pi/2} \theta^{2\lambda+1}
                \left( \frac{t^2}{t^2+(y-z)^2+2yz(1-\cos \theta)} \right)^{\lambda+1}  d\theta\right)^2 dydt\right\}^{1/2} \nonumber \\
                \leq &  C \left\{ \int_{\G_+(x)}
                \left( (yz)^\lambda \int_0^{\pi/2}
                \frac{\theta^{2\lambda+1}}{(t^2+(y-z)^2+yz\theta^2)^{\lambda+1}}   d\theta\right)^2 dydt\right\}^{1/2} \nonumber \\
                \leq & \frac{C}{z} \left( 1 + \log_+ \frac{z}{|x-z|} \right), \quad 0<x/2<z<2x.
        \end{align}
        In the last inequality we have used the estimations shown in \cite[p. 483--484]{BCFR2}. Analogously, we get
        \begin{align}\label{K11-K12}
            \|\mathcal{K}_{1,1}&(z;y,t) - \mathcal{K}_{1,2}(z;y,t)\|_{L^2(\G_+(x),\frac{dydt}{t^2})} \nonumber \\
                \leq &  C  \left\{ \int_{\G_+(x)}
                \left( \frac{(yz)^\lambda}{t^{2\lambda+2}} \int_0^{\pi/2} \theta^{2\lambda-1}
                \left| \Phi \left( \frac{(y-z)^2+2yz(1-\cos \theta)}{t^2} \right) - \Phi \left( \frac{(y-z)^2+yz\theta^2}{t^2} \right) \right|
                d\theta\right)^2 dydt\right\}^{1/2} \nonumber \\
                \leq &  C  \left\{ \int_{\G_+(x)}
                \left( \frac{(yz)^\lambda}{t^{2\lambda+2}} \int_0^{\pi/2} \theta^{2\lambda-1} \frac{yz|1-\cos\theta-\theta^2/2|}{t^2}
                \left( \frac{t^2}{t^2+(y-z)^2+yz\theta^2} \right)^{\lambda+2}  d\theta\right)^2 dydt\right\}^{1/2} \nonumber \\
                \leq &  C \left\{ \int_{\G_+(x)}
                \left( (yz)^\lambda \int_0^{\pi/2}
                \frac{\theta^{2\lambda+1}}{(t^2+(y-z)^2+yz\theta^2)^{\lambda+1}}   d\theta\right)^2 dydt\right\}^{1/2} \nonumber \\
                \leq & \frac{C}{z} \left( 1 + \log_+ \frac{z}{|x-z|} \right), \quad 0<x/2<z<2x.
        \end{align}
        We now split the kernel $\mathcal{K}_{1,2}$ as follows,
        \begin{align*}
            \mathcal{K}_{1,2}(z;y,t)
                = &  \frac{(yz)^\lambda}{\sqrt{\pi} 2^{\lambda-1/2}\G(\lambda)t^{2\lambda+1}}
                \left(\int_0^\infty - \int_{\pi/2}^\infty \right)  \theta^{2\lambda-1} \Phi \left( \frac{(y-z)^2+yz\theta^2}{t^2} \right) d\theta \\
                = & \mathcal{K}_{1,3}(z;y,t) - \mathcal{K}_{1,4}(z;y,t), \quad t,y,z \in (0,\infty).
        \end{align*}
        By making the change of variables $u=yz\theta^2/t^2$ we arrive at
        \begin{align*}
            \mathcal{K}_{1,3}(z;y,t)
                = & \frac{1}{\sqrt{\pi} 2^{\lambda+1/2}\G(\lambda)}
                \int_0^\infty  \frac{u^{\lambda-1}}{t}\Phi \left( \left(\frac{y-z}{t}\right)^2 + u \right) du
                = \Psi_t(y-z), \quad t,y,z \in (0,\infty).
        \end{align*}
        By using again that $\Phi \in S(\mathbb{R})$, the bound obtained in  \cite[p. 486--487]{BCFR2}
        allows us to write
        \begin{align}\label{K14}
            \|\mathcal{K}_{1,4}&(z;y,t)\|_{L^2(\G_+(x),\frac{dydt}{t^2})} \nonumber \\
            \leq &  C  \left\{ \int_{\G_+(x)}
                \left( \frac{(yz)^\lambda}{t^{2\lambda+2}} \int_{\pi/2}^\infty \theta^{2\lambda-1}
                 \left( \frac{t^2}{t^2+(y-z)^2+yz \theta^2} \right)^{\lambda+1}
                d\theta\right)^2 dydt\right\}^{1/2} \nonumber \\
            \leq &  C  \left\{ \int_{\G_+(x)}
                \left( (yz)^\lambda \int_{\pi/2}^\infty
                  \frac{\theta^{2\lambda-1}}{(t^2+(y-z)^2+yz \theta^2)^{\lambda+1}}
                d\theta\right)^2 dydt\right\}^{1/2}
            \leq \frac{C}{z}, \quad x,z, \in (0,\infty).
        \end{align}
        By putting together estimations \eqref{K2}--\eqref{K14}, we deduce that
        \begin{align}\label{A}
            A^+&(f \#_\lambda \phi_{(t)})(x)
                =  \left\|  \int_0^\infty f(z) \,_\lambda \tau_y(\phi_{(t)})(z) \chi_{\G_+(x)}(y,t) dz  \right\|_{\gamma\left(L^2((0,\infty)^2,\frac{dydt}{t^2});X\right)} \nonumber\\
                \leq & \int_0^\infty \left\|   f(z) \mathcal{K}_2(z;y,t) \chi_{\G_+(x)}(y,t)  \right\|_{\gamma\left(L^2((0,\infty)^2,\frac{dydt}{t^2});X\right)} dz \nonumber\\
                & + \int_{(0,x/2) \cup (2x,\infty)} \left\|   f(z) \mathcal{K}_1(z;y,t) \chi_{\G_+(x)}(y,t)  \right\|_{\gamma\left(L^2((0,\infty)^2,\frac{dydt}{t^2});X\right)} dz \nonumber \\
                & + \int_{x/2}^{2x}  \left\|   f(z) \left[\mathcal{K}_1(z;y,t) - \mathcal{K}_{1,1}(z;y,t) \right] \chi_{\G_+(x)}(y,t)  \right\|_{\gamma\left(L^2((0,\infty)^2,\frac{dydt}{t^2});X\right)} dz  \nonumber \\
                & + \int_{x/2}^{2x}  \left\|   f(z) \left[\mathcal{K}_{1,1}(z;y,t) - \mathcal{K}_{1,2}(z;y,t) \right] \chi_{\G_+(x)}(y,t)  \right\|_{\gamma\left(L^2((0,\infty)^2,\frac{dydt}{t^2});X\right)} dz \nonumber \\
                & + \int_{x/2}^{2x}  \left\|   f(z) \mathcal{K}_{1,4}(z;y,t) \chi_{\G_+(x)}(y,t)  \right\|_{\gamma\left(L^2((0,\infty)^2,\frac{dydt}{t^2});X\right)} dz \nonumber \\
                & + \left\|  \int_{x/2}^{2x} f(z) \mathcal{K}_{1,3}(z;y,t) \chi_{\G_+(x)}(y,t) dz  \right\|_{\gamma\left(L^2((0,\infty)^2,\frac{dydt}{t^2});X\right)} \nonumber \\
                \leq & \int_0^\infty    \|f(z)\|_X \left\| \mathcal{K}_2(z;y,t) \right\|_{L^2(\G_+(x),\frac{dydt}{t^2})} dz
                 + \int_{(0,x/2) \cup (2x,\infty)}  \|f(z)\|_X \left\|   \mathcal{K}_1(z;y,t)  \right\|_{L^2(\G_+(x),\frac{dydt}{t^2})} dz \nonumber \\
                & + \int_{x/2}^{2x}  \|f(z)\|_X \left\| \mathcal{K}_1(z;y,t) - \mathcal{K}_{1,1}(z;y,t) \right\|_{L^2(\G_+(x),\frac{dydt}{t^2})} dz \nonumber \\
                & + \int_{x/2}^{2x}  \|f(z)\|_X \left\| \mathcal{K}_{1,1}(z;y,t) - \mathcal{K}_{1,2}(z;y,t)   \right\|_{L^2(\G_+(x),\frac{dydt}{t^2})} dz \nonumber \\
                & + \int_{x/2}^{2x}  \|f(z)\|_X \left\| \mathcal{K}_{1,4}(z;y,t) \right\|_{L^2(\G_+(x),\frac{dydt}{t^2})} dz \nonumber \\
                & + \left\|  \int_{x/2}^{2x} f(z) \mathcal{K}_{1,3}(z;y,t) \chi_{\G_+(x)}(y,t) dz  \right\|_{\gamma\left(L^2((0,\infty)^2,\frac{dydt}{t^2});X\right)} \nonumber \\
                \leq & C \Big( H_0(\|f\|_X)(x) + H_\infty(\|f\|_X)(x) + \mathcal{N}(\|f\|_X)(x) \nonumber \\
                & + \left\|  \int_{x/2}^{2x} f(z) \mathcal{K}_{1,3}(z;y,t) \chi_{\G_+(x)}(y,t) dz  \right\|_{\gamma\left(L^2((0,\infty)^2,\frac{dydt}{t^2});X\right)} \Big),
                \quad x \in (0,\infty),
        \end{align}
        being
        $$H_0(g)(x) = \frac{1}{x} \int_0^x g(z) dz, \quad x \in (0,\infty),$$
        $$H_\infty(g)(x) =  \int_x^\infty \frac{g(z)}{z} dz, \quad x \in (0,\infty),$$
        and
        $$\mathcal{N}(g)(x)=\int_{x/2}^{2x} \frac{1}{z} \left( 1 + \log_+ \frac{z}{|x-z|}  \right) g(z) dz, \quad x \in (0,\infty).$$
        It is known that the Hardy type operators $H_0$ and $H_\infty$ are bounded from $L^p(0,\infty)$ into itself (\cite{Mu}). $\mathcal{N}$
        also maps $L^p(0,\infty)$ into $L^p(0,\infty)$, even for $p=1$. This can be easily checked by taking into account that
        $$0<\int_{1/2}^{2} \frac{1}{u} \left( 1 + \log_+ \frac{u}{|1-u|}  \right)du <\infty,$$
        and by applying Jensen's inequality.\\

        We now analyze the last quantity in \eqref{A}. We define $\tilde{f}(x)=f(x)$ when $x \geq 0$ and $\tilde{f}(x)=0$, otherwise.
        It is clear that,
        \begin{align*}
            \Big\|  \int_{x/2}^{2x} f(z) & \mathcal{K}_{1,3}(z;y,t) \chi_{\G_+(x)}(y,t) dz  \Big\|_{\gamma\left(L^2((0,\infty)^2,\frac{dydt}{t^2});X\right)}
                \leq \left\|  (\tilde{f} * \Psi_t)(y) \chi_{\G_+(x)}(y,t) \right\|_{\gamma\left(L^2((0,\infty)^2,\frac{dydt}{t^2});X\right)} \\
                & + \int_{(0,x/2) \cup (2x,\infty)} \|f(z)\|_X \left\| \Psi_t(y-z) \right\|_{L^2(\G_+(x),\frac{dydt}{t^2})} dz, \quad x \in (0,\infty),
        \end{align*}
        and this second integral is controlled by Hardy type operators. Indeed, observe that
        \begin{align}\label{second}
           \left\| \Psi_t(y-z) \right\|_{L^2(\G_+(x),\frac{dydt}{t^2})}
            \leq & C \left\{ \int_0^\infty \int_{|x-y|}^\infty \frac{dydt}{(t+|y-z|)^4} \right\}^{1/2} \nonumber \\
            \leq & C  \left(\int_{I_{x,z}} \frac{dy}{(|y-z|+|x-y|)^{3}} + \int_{\mathbb{R} \setminus I_{x,z}} \frac{dy}{(|y-z|+|x-y|)^{3}} \right)^{1/2} \nonumber\\
            \leq & \frac{C}{|x-z|}
            \leq C \left\{ \begin{array}{rl}
                        \dfrac{1}{x}, & 0 < z \leq x/2,\\
                        & \\
                        \dfrac{1}{z}, & 0 < 2x \leq z,\\
                     \end{array} \right.
        \end{align}
        because $\Psi \in S(\mathbb{R})$. Here $I_{x,z}$ represents the interval $(\min\{x,z\},\max\{x,z\})$.
        Hence, it only remains to prove that
        \begin{equation*}
            \left\| \left\|  (\tilde{f} * \Psi_t)(y) \chi_{\G_+(x)}(y,t) \right\|_{\gamma\left(L^2((0,\infty)^2,\frac{dydt}{t^2});X\right)} \right\|_{L^p(0,\infty)}
                \leq C \|f\|_{L^p((0,\infty),X)}.
        \end{equation*}
        Let $\{\gamma_j\}_{j \in \mathbb{N}}$ be a sequence of independent complex standard Gaussian random variables in the probability space
        $(\Omega, \mathcal{A},\mathbb{P})$. Suppose that $\{h_j\}_{j=1}^N$ is a finite family of orthonormal functions in $L^2((0,\infty)^2,\frac{dydt}{t^2})$.
        By defining
        $$\tilde{h}_j(y,t)
            = \left\{ \begin{array}{rl}
                        h_j(y,t), & y >0, \ t>0,\\
                        & \\
                        0, & y \leq 0, \ t>0,\\
                     \end{array} \right. $$
        $\{\tilde{h}_j\}_{j=1}^N$ is also orthonormal in $L^2(\mathbb{R}^2_+,\frac{dydt}{t^2})$, where $\mathbb{R}^2_+$ denotes the half-space
        $\mathbb{R} \times (0,\infty)$. Then,
        \begin{align} \label{gamma}
            & \left\{ \int_\Omega \left\| \sum_{j=1}^N \gamma_j(w)
            \int_{(0,\infty)^2} (\tilde{f} * \Psi_t)(y) \chi_{\G_+(x)}(y,t) h_j(y,t) \frac{dydt}{t^2} \right\|_X^2 d\mathbb{P}(w) \right\}^{1/2} \nonumber\\
            & \qquad =  \left\{ \int_\Omega \left\| \sum_{j=1}^N \gamma_j(w)
            \int_{\mathbb{R}^2_+} (\tilde{f} * \Psi_t)(y) \chi_{\G(x)}(y,t) \tilde{h}_j(y,t) \frac{dydt}{t^2} \right\|_X^2 d\mathbb{P}(w) \right\}^{1/2} \nonumber \\
            &\qquad \leq  \left\|  (\tilde{f} * \Psi_t)(y) \chi_{\G(x)}(y,t) \right\|_{\gamma\left(L^2(\mathbb{R}^2_+,\frac{dydt}{t^2});X\right)}, \quad x \in (0,\infty).
        \end{align}
        Here $\G(x)=\{(y,t) \in \mathbb{R}^2_+ : |x-y|<t\}$, $x \in \mathbb{R}$.\\

        Finally, applying \cite[Theorem 4.2]{HW} we conclude that
        \begin{align*}
            \left\| \left\|  (\tilde{f} * \Psi_t)(y) \chi_{\G_+(x)}(y,t) \right\|_{\gamma\left(L^2((0,\infty)^2,\frac{dydt}{t^2});X\right)} \right\|_{L^p(0,\infty)}
               \leq &  \left\| \left\|  (\tilde{f} * \Psi_t)(y) \chi_{\G(x)}(y,t) \right\|_{\gamma\left(L^2(\mathbb{R}^2_+,\frac{dydt}{t^2});X\right)} \right\|_{L^p(\mathbb{R})} \\
               \leq & C \|\tilde{f}\|_{L^p(\mathbb{R},X)}
               = C \|f\|_{L^p((0,\infty),X)}.
        \end{align*}
        Note that, \cite[Theorem 4.2]{HW} requires that $\Psi$ has vanishing integral, and this is a consequence of
        the hypothesis imposed over $\phi$, as we now show,
        \begin{align}\label{nula}
            \int_\mathbb{R} \Psi(w) dw
                = & \frac{1}{\sqrt{\pi} 2^{\lambda-1/2} \G(\lambda)} \int_0^\infty \int_0^\infty u^{\lambda-1} \frac{\phi(\sqrt{w^2+u})}{\left( \sqrt{w^2+u} \right)^\lambda} du dw \nonumber \\
                = & \frac{1}{\sqrt{\pi} 2^{\lambda-3/2} \G(\lambda)} \int_0^\infty \int_w^\infty (z^2-w^2)^{\lambda-1} z^{1-\lambda} \phi(z) dz dw \nonumber \\
                = & \frac{1}{\sqrt{\pi} 2^{\lambda-3/2} \G(\lambda)} \int_0^\infty z^{1-\lambda} \phi(z)\int_0^z (z^2-w^2)^{\lambda-1}    dw dz \nonumber \\
                = & \frac{1}{ 2^{\lambda-1/2} \G(\lambda+1/2)} \int_0^\infty z^{\lambda}  \phi(z) dz =0.
        \end{align}
    \end{proof}

    We now introduce the operator $\mathcal{Y}_\lambda$ defined by
    $$\mathcal{Y}_\lambda(f)(x,t)
        =  \int_0^\infty f(y) t D_{\lambda,y} P_t^\lambda(x,y) dy, \quad t,x \in (0,\infty),$$
    where $D_{\lambda,y}=y^\lambda \frac{d}{dy} y^{-\lambda}$. Note that $\mathcal{Y}_\lambda$ is not a Hankel
    convolution operator. Hence, the next result is not a special case of Lemma~\ref{Lem_Lp}.

    \begin{Lem}\label{Lem_Lp_dx}
        Let $X$ be a UMD Banach space, $\lambda>0$ and $1<p<\infty$. Then, there exists $C>0$ such that
        $$\left\|A^+\left( \mathcal{Y}_\lambda(f)  \right) \right\|_{L^p(0,\infty)}
            \leq C \|f\|_{L^p((0,\infty),X)}, \quad f \in L^p((0,\infty),X).$$
    \end{Lem}

    \begin{proof}
        Let $f \in L^p((0,\infty),X)$. We have that
        \begin{align*}
           t D_{\lambda,z} P_t^\lambda(y,z)
            = & - \frac{4\lambda (\lambda+1) }{\pi} t^2 (yz)^\lambda \left(\int_0^{\pi/2} + \int_{\pi/2}^\pi  \right) \frac{(\sin \theta)^{2\lambda-1}[(z-y)+y(1-\cos \theta)]}{[(y-z)^2+t^2+2yz(1-\cos \theta)]^{\lambda+2}} d\theta\\
            = & \mathcal{L}_1(z;y,t) + \mathcal{L}_2(z;y,t), \quad t,y,z \in (0,\infty).
        \end{align*}
        We introduce the following kernels, which will help us to obtain the desired estimations,
        $$\mathcal{L}_{1,1}(z;y,t)
            = - \frac{4\lambda (\lambda+1) }{\pi} t^2 (yz)^\lambda \int_0^{\pi/2} \frac{ \theta^{2\lambda-1}[(z-y)+y(1-\cos \theta)]}{[(y-z)^2+t^2+2yz(1-\cos \theta)]^{\lambda+2}} d\theta , $$
        \begin{align*}
            \mathcal{L}_{1,2}(z;y,t)
                = & - \frac{4\lambda (\lambda+1) }{\pi} t^2 (yz)^\lambda \left( \int_0^\infty - \int_{\pi/2}^\infty \right) \frac{ \theta^{2\lambda-1}[(z-y)+y\theta^2/2]}{[(y-z)^2+t^2+yz\theta^2]^{\lambda+2}} d\theta \\
                = & \mathcal{L}_{1,3}(z;y,t) - \mathcal{L}_{1,4}(z;y,t),
        \end{align*}
        $$\mathcal{L}_{1,5}(z;y,t)
            = - \frac{4\lambda (\lambda+1) }{\pi} t^2 (yz)^\lambda \int_0^\infty \frac{ \theta^{2\lambda-1}(z-y)}{[(y-z)^2+t^2+yz\theta^2]^{\lambda+2}} d\theta , $$
        and $\mathcal{L}_{1,6}(z;y,t)=\mathcal{L}_{1,3}(z;y,t)-\mathcal{L}_{1,5}(z;y,t)$, for every $t,y,z \in (0,\infty)$.\\

        We can write
        \begin{align}\label{A+}
            A^+&\left( \int_0^\infty f(z) t D_{\lambda,z} P_t^\lambda(y,z) dz \right)(x)
                \leq  \int_0^\infty    \|f(z)\|_X \left\| \mathcal{L}_2(z;y,t) \right\|_{L^2(\G_+(x),\frac{dydt}{t^2})} dz \nonumber \\
                & + \int_{(0,x/2) \cup (2x,\infty)}  \|f(z)\|_X \left\|   \mathcal{L}_1(z;y,t)  \right\|_{L^2(\G_+(x),\frac{dydt}{t^2})} dz  \nonumber \\
                & + \int_{x/2}^{2x}  \|f(z)\|_X \left\| \mathcal{L}_1(z;y,t) - \mathcal{L}_{1,1}(z;y,t) \right\|_{L^2(\G_+(x),\frac{dydt}{t^2})} dz  \nonumber \\
                & + \int_{x/2}^{2x}  \|f(z)\|_X \left\| \mathcal{L}_{1,1}(z;y,t) - \mathcal{L}_{1,2}(z;y,t)   \right\|_{L^2(\G_+(x),\frac{dydt}{t^2})} dz  \nonumber \\
                & + \int_{x/2}^{2x}  \|f(z)\|_X \left\| \mathcal{L}_{1,4}(z;y,t) - \mathcal{L}_{1,6}(z;y,t) \right\|_{L^2(\G_+(x),\frac{dydt}{t^2})} dz \nonumber \nonumber \\
                & + \left\|  \int_{x/2}^{2x} f(z) \mathcal{L}_{1,5}(z;y,t) \chi_{\G_+(x)}(y,t) dz  \right\|_{\gamma\left(L^2((0,\infty)^2,\frac{dydt}{t^2});X\right)}
                , \quad x \in (0,\infty).
        \end{align}
        Our objective is to analyze the $L^p$-boundedness properties of all the operators appearing in each line
        in the right hand side of the last inequality.\\

        First of all, by \eqref{K2} we have that
        \begin{align*}
            \left\| \mathcal{L}_2(z;y,t) \right\|_{L^2(\G_+(x),\frac{dydt}{t^2})}
                & \leq C z^\lambda \left\{\int_{\G_+(x)}  \frac{ dt dy}{(t+y+z)^{2\lambda+4}} \right\}^{1/2}
                  \leq \frac{C}{x+z}, \quad x,z \in (0,\infty),
        \end{align*}
        and also by proceeding as in \eqref{K1_global},
        \begin{align*}
            \left\| \mathcal{L}_1(z;y,t) \right\|_{L^2(\G_+(x),\frac{dydt}{t^2})}
                \leq &  C \left\{ \int_{\G_+(x)} \left| t (yz)^\lambda \int_0^{\pi/2}\frac{(\sin \theta)^{2\lambda-1}[z(1-\cos \theta) -(y-z)\cos \theta)}{[(y-z)^2+t^2+2yz(1-\cos \theta)]^{\lambda+2}} d\theta \right|^2 dy dt \right\}^{1/2} \\
                \leq &  C \Big( \left\{ \int_{\G_+(x)} \frac{z^{2\lambda+2}y^{2\lambda}}{(|y-z|+t)^{4\lambda+6}}  dy dt \right\}^{1/2}
                     +  \left\{ \int_{\G_+(x)} \frac{z^{2\lambda}y^{2\lambda}}{(|y-z|+t)^{4\lambda+4}}  dy dt \right\}^{1/2} \Big)   \\
                \leq &  C \Big( \frac{z^{\lambda+1}}{|x-z|^{\lambda+2}} + \frac{z^{\lambda}}{|x-z|^{\lambda+1}} \Big)
                \leq C \left\{ \begin{array}{rl}
                        \dfrac{1}{x}, & 0 < z \leq x/2,\\
                        & \\
                        \dfrac{1}{z}, & 0 < 2x \leq z.\\
                     \end{array} \right.
        \end{align*}
        Then, the two first operators in \eqref{A+} are controlled by the Hardy type operators $H_0$ and $H_\infty$, which are bounded in $L^p(0,\infty)$.\\

        By applying the mean value theorem we obtain, for every $t,y,z \in (0,\infty)$,
        \begin{align*}
            \Big| \mathcal{L}_1(z;y,t) &- \mathcal{L}_{1,1}(z;y,t) \Big| +  \left| \mathcal{L}_{1,1}(z;y,t) - \mathcal{L}_{1,2}(z;y,t) \right| \\
                 &\leq  C \left( t^2 (yz)^\lambda \int_0^{\pi/2} \frac{ \theta^{2\lambda+3}y}{[(y-z)^2+t^2+yz\theta^2]^{\lambda+2}} d\theta
                         + t(yz)^\lambda \int_0^{\pi/2} \frac{ \theta^{2\lambda+1}}{[(y-z)^2+t^2+yz\theta^2]^{\lambda+1}} d\theta  \right).
        \end{align*}
        Moreover, we have that
        \begin{align} \label{GB}
            \left| \mathcal{L}_{1,6}(z;y,t) - \mathcal{L}_{1,4}(z;y,t) \right|
                \leq & C \Big( t^2 (yz)^\lambda \int_0^{\pi/2} \frac{ \theta^{2\lambda+1}y}{[(y-z)^2+t^2+yz\theta^2]^{\lambda+2}} d\theta \nonumber \\
                     &   + t(yz)^\lambda \int_{\pi/2}^\infty \frac{ \theta^{2\lambda-1}}{[(y-z)^2+t^2+yz\theta^2]^{\lambda+1}} d\theta \Big),
                \quad t,y,z \in (0,\infty).
        \end{align}
        By using \eqref{K1-K11}, \eqref{K11-K12} and \eqref{K14} it follows that the third, fourth and fifth operator
        appearing in \eqref{A+} are bounded from $L^p((0,\infty),X)$ into $L^p(0,\infty)$ provided that the operator
        $$\mathcal{T}_\lambda(g)(x)
            = \int_{x/2}^{2x}  g(z) \left\| \mathcal{I}_\lambda(z;y,t) \right\|_{L^2(\G_+(x),\frac{dydt}{t^2})} dz, \quad x \in (0,\infty),$$
        is bounded  from $L^p(0,\infty)$ into itself, where
        \begin{equation}\label{GC}
            \mathcal{I}_\lambda(z;y,t)
                = t^2 (yz)^\lambda \int_0^{\pi/2} \frac{ \theta^{2\lambda+1}y}{[(y-z)^2+t^2+yz\theta^2]^{\lambda+2}} d\theta,
                \quad t,y,z \in (0,\infty).
        \end{equation}

        In order to study the kernel $\mathcal{I}_\lambda$ we distingue three cases. Firstly, assume that
        $0<y \leq z/2$. We can write
        \begin{align*}
            |\mathcal{I}_\lambda(z;y,t)|
                \leq & C t (yz)^\lambda \int_0^{\pi/2} \frac{ \theta^{2\lambda+1}}{[(y-z)^2+t^2+yz\theta^2]^{\lambda+1}} \frac{ty}{z^2+t^2+yz\theta^2} d\theta \\
                \leq & C t (yz)^\lambda \int_0^{\pi/2} \frac{ \theta^{2\lambda+1}}{[(y-z)^2+t^2+yz\theta^2]^{\lambda+1}} d\theta,
                \quad t \in (0,\infty).
        \end{align*}
        In a similar way, if $0<2z \leq y$, we get
        $$|\mathcal{I}_\lambda(z;y,t)|
            \leq C t (yz)^\lambda \int_0^{\pi/2} \frac{ \theta^{2\lambda+1}}{[(y-z)^2+t^2+yz\theta^2]^{\lambda+1}} d\theta,
                \quad t \in (0,\infty).$$
        Moreover, if $z/2<y<2z$, then
        \begin{align*}
            |\mathcal{I}_\lambda(z;y,t)|
                \leq & C t \int_0^{\pi/2} \frac{ (z\theta)^{2\lambda+1}}{[(y-z)^2+t^2+(z\theta)^2]^{\lambda+3/2}} d\theta
                \leq  C t \int_0^\infty \frac{ d\theta}{(|y-z|+t+z\theta)^{2}} \\
                \leq & C \frac{ t}{z(|y-z|+t)},  \quad t \in (0,\infty).
        \end{align*}
        From \eqref{K1-K11} we deduce that, for each $0<x/2<z<2x<\infty$,
        \begin{align*}
            \left( \int_{\G_+(x)}  |\mathcal{I}_\lambda(z;y,t)|^2 \frac{dydt}{t^2} \right)^{1/2}
                \leq & \frac{C}{z} \left( 1 + \log_+ \frac{z}{|x-z|}
                    + \left( \int_{z/2}^{2z} \int_{|x-y|}^\infty \frac{dtdy}{(|y-z|+t)^2} \right)^{1/2} \right) \\
                \leq & \frac{C}{z} \left( 1 + \log_+ \frac{z}{|x-z|}
                    + \left( \int_{z/2}^{2z} \frac{dy}{|y-z|+|x-y|} \right)^{1/2} \right) \\
                \leq & \frac{C}{z} \left( 1 + \log_+ \frac{z}{|x-z|}
                    + \left(\frac{z}{|x-z|} \right)^{1/2} \right).
        \end{align*}
        Hence, the operator $\mathcal{T}_\lambda$ is bounded from $L^p(0,\infty)$ into itself.\\

        To finish the proof of this lemma we have to show that the operator $\mathcal{Z}_\lambda$ defined by
        $$\mathcal{Z}_\lambda(f)(x;y,t)
            = \int_{x/2}^{2x} f(z) \mathcal{L}_{1,5}(z;y,t) \chi_{\G_+(x)}(y,t) dz$$
        is bounded from $L^p((0,\infty),X)$ into $L^p\left((0,\infty);\gamma\left(L^2((0,\infty)^2,\frac{dydt}{t^2});X\right)\right)$.\\

         The change of variables $\theta=u\sqrt{((y-z)^2+t^2)/yz}$ gives us
        \begin{align}\label{G1}
            \mathcal{L}_{1,5}(z;y,t)
                = & \frac{2t^2}{\pi} \frac{y-z}{[(y-z)^2+t^2]^2}
                = t \partial_z (P_t(y-z)), \quad t,y,z \in (0,\infty),
        \end{align}
        because $\int_0^\infty u^{2\lambda-1}/(1+u^2)^{\lambda+2} du = 1/(2\lambda(\lambda+1))$.
        Here $P_t(z)$ represents the kernel of the
        Poisson semigroup associated  with the Euclidean Laplacian,
        $$P_t(z)= \frac{1}{\pi} \frac{t}{t^2+z^2}, \quad z \in \mathbb{R}, \ t \in (0,\infty).$$

        We define $\tilde{f}(z)=f(z)$, $z \geq 0$, and $\tilde{f}(z)=0$ when $z <0$.
        Then, \eqref{gamma} leads to
        \begin{align*}
            \Big\|  \mathcal{Z}_\lambda(f)(x; \cdot , \cdot)  \Big\|_{\gamma\left(L^2((0,\infty)^2,\frac{dydt}{t^2});X\right)}
                \leq & \left\|  (\tilde{f} * t \partial_z P_t )(y) \chi_{\G(x)}(y,t) \right\|_{\gamma\left(L^2(\mathbb{R}^2_+,\frac{dydt}{t^2});X\right)} \\
                & + \int_{(0,x/2) \cup (2x,\infty)} \|f(z)\|_X \left\| t \partial_z P_t(y-z) \right\|_{L^2(\G_+(x),\frac{dydt}{t^2})} dz, \ x \in (0,\infty).
        \end{align*}
        Also, by doing the same computation as in \eqref{second}, it follows that
        \begin{align*}
            \int_{(0,x/2) \cup (2x,\infty)} & \|f(z)\|_X \left\| t \partial_z P_t(y-z) \right\|_{L^2(\G_+(x),\frac{dydt}{t^2})} dz
                \leq C \Big( H_0(\|f\|_X)(x) + H_\infty(\|f\|_X)(x) \Big), \ x \in (0,\infty).
        \end{align*}
        Note that $t \partial_z P_t(z)=h_t(z)$, $z \in \mathbb{R}$ and $t>0$, where
        $h(z)=-\frac{2}{\pi} \frac{z}{(z^2+1)^2}$, $z \in \mathbb{R}$. It is clear that $\int_\mathbb{R} h(z) dz=0$
        and that $h \notin S(\mathbb{R})$, and then \cite[Theorem 4.2]{HW} cannot be apply directly.
        We define the operator $\mathcal{S}$ by
        $$\mathcal{S}(g)(y,t)
            = \int_{\mathbb{R}} t \partial_z P_t(y-z) g(z) dz, \quad g \in L^2(\mathbb{R}). $$
        It is well-known that
        \begin{equation}\label{sq}
            \left\| \left( \int_{\G(x)} |\mathcal{S}(g)(y,t)|^2 \frac{dy dt}{t^2} \right)^{1/2} \right\|_{L^q(\mathbb{R})}
            \leq C \|g\|_{L^q(\mathbb{R})}, \quad g \in L^q(\mathbb{R}), \ 1<q<\infty.
        \end{equation}
        (see \cite[p. 27--28, 180--182]{Ste}). Moreover, the kernel $t \partial_z P_t(y-z)$ satisfies
        the hypothesis in \cite[Theorem 4.8]{HNP}. Hence, there exists $C>0$ such that for every
        $g \in L^p(\mathbb{R}) \otimes X$,
        $$\left\| \left\|  (g * t \partial_z P_t )(y) \chi_{\G(x)}(y,t) \right\|_{\gamma\left(L^2(\mathbb{R}^2_+,\frac{dydt}{t^2});X\right)} \right\|_{L^p(\mathbb{R})}
            \leq C \|g\|_{L^p(\mathbb{R},X)}.$$

        We denote by $\tilde{\mathcal{S}}$ the extension of the operator $\mathcal{S}\otimes I_X$ to
        $L^p(\mathbb{R},X)$ as a bounded operator from $L^p(\mathbb{R},X)$ to $T^{p,2}\left(\mathbb{R},\gamma\left(L^2(\mathbb{R}^2_+,\frac{dydt}{t^2});X\right)\right)$,
        where the tent space $T^{p,2}\left(\mathbb{R},\gamma\left(L^2(\mathbb{R}^2_+,\frac{dydt}{t^2});X\right)\right)$ is the completion of
        $C_c(\mathbb{R}^2_+)\otimes X$ with respect to the norm
        $$ \|F\|_{T^{p,2}\left(\mathbb{R},\gamma\left(L^2(\mathbb{R}^2_+,\frac{dydt}{t^2});X\right)\right)}
            = \| F(y,t)\chi_{\G(x)}(y,t)\|_{L^p\left(\mathbb{R},\gamma\left(L^2(\mathbb{R}^2_+,\frac{dydt}{t^2});X\right)\right)}.$$

        Our next objective is to show that $\tilde{\mathcal{S}}(g)=\mathcal{S}(g)$,
        $g \in L^p(\mathbb{R},X)$, that is,
        $$\tilde{\mathcal{S}}(g)(y,t)
            = \int_{\mathbb{R}} t \partial_z P_t(y-z) g(z) dz, \quad g \in L^p(\mathbb{R},X), $$
        where the last integral is understood in the Bochner's sense.\\

        Let $g \in L^p(\mathbb{R},X)$. Observe that $F(y,t)=\mathcal{S}(g)(y,t) \chi_{\G(x)}(y,t)$ is weakly-$L^2(\mathbb{R}_+^2,\frac{dydt}{t^2};X)$
        for almost every $x \in \mathbb{R}$. Indeed, from \eqref{sq} we deduce that
        \begin{align*}
            \left\| \left( \int_{\G(x)} |\langle F(y,t),x'\rangle|^2 \frac{dy dt}{t^2} \right)^{1/2} \right\|_{L^p(\mathbb{R})}
                = & \left\| \left( \int_{\G(x)} |\mathcal{S}(\langle g,x'\rangle)(y,t)|^2 \frac{dy dt}{t^2} \right)^{1/2} \right\|_{L^p(\mathbb{R})}\\
                \leq & C \|\langle g,x'\rangle\|_{L^p(\mathbb{R})}
                \leq  C \|x'\|_{X'}\| g\|_{L^p(\mathbb{R},X)}<\infty, \quad x' \in X',
        \end{align*}
        and this implies that $\langle F(y,t),x'\rangle \in L^2(\mathbb{R}_+^2,\frac{dydt}{t^2})$, a.e. $x \in \mathbb{R}$.\\

        From \eqref{G1} it follows that
        $$| t \partial_z P_t(y-z)|
            \leq C \frac{t}{(|y-z|+t)^2}, \quad t \in (0,\infty), \ y,z \in \mathbb{R}.$$
        If $N \in \mathbb{N}$, we denote $\G_N(x)=\{(y,t) \in \mathbb{R}^2_+ :  |x-y|<t, 1/N<t<N\}$. We have that
        \begin{align*}
            \left\| \mathcal{S}(g)(y,t) \chi_{\G_N(x)}(y,t) \right\|^2_{L^2\left( \mathbb{R}_+^2, \frac{dydt}{t^2}; X\right)}
                \leq & C \int_{\G_N(x)} \left( \int_\mathbb{R} \frac{\|g(z)\|_X}{(|y-z|+t)^2}  dz\right)^2 dy dt \\
                \leq & C \|g\|^2_{L^p(\mathbb{R},X)} \int_{\G_N(x)} \left( \int_\mathbb{R} \frac{dz}{(|y-z|+t)^{2p'}}  \right)^{2/p'} dy dt \\
                \leq & C \|g\|^2_{L^p(\mathbb{R},X)}, \quad N \in \mathbb{N}, \ x \in \mathbb{R}.
        \end{align*}
        Here $C>0$ depends on $N \in \mathbb{N}$ but it does not depend on g.
        Assume that $g_n \in L^p(\mathbb{R}) \otimes X$, $n \in \mathbb{N}$, and that $g_n \longrightarrow g$,
        as $n \to \infty$, in $L^p(\mathbb{R},X)$. Then, for every $N \in \mathbb{N}$ and $x \in \mathbb{R}$,
        $$\mathcal{S}(g_n) \chi_{\G_N(x)} \longrightarrow \mathcal{S}(g) \chi_{\G_N(x)},
            \quad \text{as } n \to \infty, \text{ in } L^2\left( \mathbb{R}_+^2, \frac{dydt}{t^2}; X\right).$$
        Also,
        $$\mathcal{S}(g_n) \chi_{\G(x)} \longrightarrow \tilde{\mathcal{S}}(g) \chi_{\G(x)},
            \quad \text{as } n \to \infty, \text{ in } L^p\left(\mathbb{R}; \gamma\left( L^2\left( \mathbb{R}_+^2, \frac{dydt}{t^2}\right); X \right)\right).$$
        Hence, there exists a subset $\Omega$ of $\mathbb{R}$ such that $|\mathbb{R}\setminus \Omega|=0$,
        and an increasing sequence $\{m_k\}_{k=1}^\infty \subset \mathbb{N}$, such that, for every $x \in \Omega$,
        $$\mathcal{S}(g_{m_k}) \chi_{\G(x)} \longrightarrow \tilde{\mathcal{S}}(g) \chi_{\G(x)},
            \quad \text{as } k \to \infty, \text{ in }  \gamma\left( L^2\left( \mathbb{R}^2_+, \frac{dydt}{t^2}\right); X \right).$$
        Therefore, if $\mathcal{L}\left( L^2\left( \mathbb{R}_+^2, \frac{dydt}{t^2}\right) ,X \right)$ denotes
        the space of bounded operators from $L^2\left( \mathbb{R}_+^2, \frac{dydt}{t^2}\right)$ to $X$, we have that,
        for every $x \in \Omega$,
         $$\mathcal{S}(g_{m_k}) \chi_{\G(x)} \longrightarrow \tilde{\mathcal{S}}(g) \chi_{\G(x)},
            \quad \text{as } k \to \infty, \text{ in }  \mathcal{L}\left(L^2\left( \mathbb{R}^2_+, \frac{dydt}{t^2}\right), X \right).$$
        Let $x \in \Omega$. Suppose that $h \in L^2(\mathbb{R}_+^2, \frac{dydt}{t^2})$ such that $\supp h \subset \mathbb{R}^2_+$
        is compact. We can write
        \begin{align*}
            \left[ \tilde{\mathcal{S}}(g) \chi_{\G(x)} \right](h)
                = & \lim_{k \to \infty} \int_{\G(x)} \mathcal{S}(g_{m_k})(y,t) h(y,t) \frac{dydt}{t^2}
                = \int_{\G(x)} \mathcal{S}(g)(y,t) h(y,t) \frac{dydt}{t^2}.
        \end{align*}
        Then, we deduce that
        $$ \tilde{\mathcal{S}}(g) \chi_{\G(x)}
            =  \mathcal{S}(g) \chi_{\G(x)} .$$
        Thus, the proof of our result is finished.
    \end{proof}

    We now define the vector valued version of the Hardy space $H^1_o(\mathbb{R})$ introduced by Fridli \cite{Fr}. We say that a strongly measurable
    $X$-valued function $a$ defined on $(0,\infty)$ is an $o$-atom when satisfies one of the following two conditions:
    \begin{itemize}
         \item $a=b\chi_{(0,\delta)}/\delta,$ where $\delta>0$ and $b \in X$ with $\|b\|_X=1$.
        \item There exists a bounded interval $I \subset (0,\infty)$ such that $\supp(a) \subset I$, $\int_I a(x)dx=0$ and
        $\|a\|_{L^\infty((0,\infty),X)} \leq 1/|I|$.
    \end{itemize}

    A strongly measurable $X$-valued odd function $f$ defined on $\mathbb{R}$
    is in $H^1_o(\mathbb{R},X)$ when $f \chi_{(0,\infty)}=\sum_{j=1}^\infty \lambda_j a_j$,
    where, for every $j \in \mathbb{N}$, $a_j$ is an $o$-atom and $\lambda_j \in \mathbb{C}$ being $\sum_{j=1}^\infty |\lambda_j|<\infty$. As usual,
    the norm $\|f\|_{H^1_o(\mathbb{R},X)}$ of $f \in H^1_o(\mathbb{R},X)$ is defined by
    $$\|f\|_{H^1_o(\mathbb{R},X)} = \inf \sum_{j=1}^\infty |\lambda_j|,$$
    where the infimum is taken over all possible sequences $\{\lambda_j\}_{j = 1}^\infty \subset \mathbb{C}$ such that $\sum_{j=1}^\infty |\lambda_j|<\infty$
    and $f \chi_{(0,\infty)}=\sum_{j=1}^\infty \lambda_j a_j$ for a certain family of $\{a_j\}_{j=1}^\infty$ of $o$-atoms. Note that
    $$\|f\|_{L^1((0,\infty),X)} \leq \|f\|_{H^1_o(\mathbb{R},X)}, \quad f \in H^1_o(\mathbb{R},X).$$

    \begin{Lem}\label{Lem_L1}
        Let $X$ be a UMD Banach space and $\lambda>0$. Assume that $\phi \in S_\lambda(0,\infty)$
        is such that $x^\lambda \phi$ has vanishing integral over $(0,\infty)$. Then, there exists $C>0$ for which
        $$\|A^+(f \#_\lambda \phi_{(t)})\|_{L^1(0,\infty)}
            \leq C \|f\|_{H^1_o(\mathbb{R},X)}, \quad f \in H^1_o(\mathbb{R},X).$$
    \end{Lem}

    \begin{proof}
        Let $f \in H^1_o(\mathbb{R},X)$. In the proof of Lemma~\ref{Lem_Lp} it was shown that
        \begin{align*}
            A^+(f \#_\lambda \phi_{(t)})(x)
                \leq & C \Big(\int_0^\infty    \|f(z)\|_X \left\| \mathcal{K}_2(z;y,t) \right\|_{L^2(\G_+(x),\frac{dydt}{t^2})} dz \\
                & + \int_{(0,x/2) \cup (2x,\infty)}  \|f(z)\|_X \left\|   \mathcal{K}_1(z;y,t)  \right\|_{L^2(\G_+(x),\frac{dydt}{t^2})} dz
                  + \mathcal{N}(\|f\|_X)(x) \\
                & + \left\|  \int_{x/2}^{2x} f(z) \Psi_t(y-z)  dz  \chi_{\G_+(x)}(y,t)\right\|_{\gamma\left(L^2((0,\infty)^2,\frac{dydt}{t^2});X\right)} \Big),
                \quad x \in (0,\infty).
        \end{align*}
        We analyze the $L^1$-boundedness of each operator in detail.
        As it was mentioned in the proof of Lemma~\ref{Lem_Lp}, $\mathcal{N}$ maps $L^1(0,\infty)$ into itself.
        The estimate \eqref{K2} allows us to write
        \begin{align*}
            \Big\|  \int_0^\infty  &  \|f(z)\|_X  \left\| \mathcal{K}_2(z;y,t) \right\|_{L^2(\G_+(x),\frac{dydt}{t^2})} dz \Big\|_{L^1(0,\infty)}
                \leq C \int_0^\infty \int_0^\infty \|f(z)\|_X \frac{z^\lambda}{(x+z)^{\lambda+1}}  dz dx\\
                = & C \int_0^\infty \|f(z)\|_X dz  \int_0^\infty \frac{dw}{(1+w)^{\lambda+1}}
                = C \|f\|_{L^1((0,\infty),X)}.
        \end{align*}
        Also by \eqref{K1_global}, we obtain
        \begin{align} \label{7.1}
            \Big\|  \int_{(0,x/2) \cup (2x,\infty)} & \|f(z)\|_X  \left\|   \mathcal{K}_1(z;y,t)  \right\|_{L^2(\G_+(x),\frac{dydt}{t^2})} dz  \Big\|_{L^1(0,\infty)}
                \leq C \int_0^\infty \int_{(0,x/2) \cup (2x,\infty)} \|f(z)\|_X \frac{z^\lambda}{|x-z|^{\lambda+1}} dz dx  \nonumber \\
                = & C \int_0^\infty \|f(z)\|_X \left( \int_0^{z/2} \frac{z^\lambda}{|x-z|^{\lambda+1}} dx + \int_{2z}^\infty \frac{z^\lambda}{|x-z|^{\lambda+1}} dx  \right) dz
                \leq C \|f\|_{L^1((0,\infty),X)}.
        \end{align}
        Since $f$ is an odd function, it has that
        \begin{align*}
            \int_{x/2}^{2x}  &f(z)\Psi_t(y-z) dz
                 = (f * \Psi_t)(y) - \left( \int_0^\infty f(z) \left[\Psi_t(y-z) - \Psi_t(y+z) \right] dz - \int_{x/2}^{2x} f(z) \Psi_t(y-z) dz \right) \\
                & = (f * \Psi_t)(y) - \left( \int_{(0,x/2)\cup(2x,\infty)} f(z) \left[\Psi_t(y-z) - \Psi_t(y+z) \right] dz - \int_{x/2}^{2x} f(z) \Psi_t(y+z) dz \right), \ x \in (0,\infty).
        \end{align*}
        Thus,
        \begin{align}\label{8.1}
            \Big\|  \int_{x/2}^{2x} f(z) & \Psi_t(y-z)  dz  \chi_{\G_+(x)}(y,t)\Big\|_{\gamma\left(L^2((0,\infty)^2,\frac{dydt}{t^2});X\right)}
                \leq \Big\| (f * \Psi_t)(y)\chi_{\G_+(x)}(y,t)  \Big\|_{\gamma\left(L^2((0,\infty)^2,\frac{dydt}{t^2});X\right)} \nonumber \\
                & + \int_{(0,x/2)\cup(2x,\infty)} \|f(z)\|_X \left\|\Psi_t(y-z) - \Psi_t(y+z) \right\|_{L^2(\G_+(x),\frac{dydt}{t^2})} dz \nonumber \\
                & + \int_{x/2}^{2x} \|f(z)\|_X \left\|\Psi_t(y+z) \right\|_{L^2(\G_+(x),\frac{dydt}{t^2})} dz, \quad x \in (0,\infty).
        \end{align}
        Now, we apply \eqref{gamma} and \cite[Corollary 4.3]{HW},
        \begin{align*}
            \Big\| \Big\| (f * \Psi_t)(y)  \chi_{\G_+(x)}(y,t)  \Big\|_{\gamma\left(L^2((0,\infty)^2,\frac{dydt}{t^2});X\right)} \Big\|_{L^1(0,\infty)}
                \leq &\Big\| \Big\| (f * \Psi_t)(y)\chi_{\G(x)}(y,t)  \Big\|_{\gamma\left(L^2(\mathbb{R}^2_+,\frac{dydt}{t^2});X\right)} \Big\|_{L^1(\mathbb{R})} \\
                \leq & C \|f\|_{H^1(\mathbb{R},X)}
                \leq  C \|f\|_{H^1_o(\mathbb{R},X)}.
        \end{align*}
        Finally, taking into account that $\Psi \in S(\mathbb{R})$ and proceeding as above we can see that
        \begin{equation*}
            \left\|\Psi_t(y+z) \right\|_{L^2(\G_+(x),\frac{dydt}{t^2})}
                \leq \frac{C}{x+z}, \quad x,z \in (0,\infty),
        \end{equation*}
        and also, by the mean value theorem,
        \begin{equation*}
            \left\|\Psi_t(y-z) - \Psi_t(y+z) \right\|_{L^2(\G_+(x),\frac{dydt}{t^2})}
                \leq C \frac{z}{|x-z|^2}, \quad x,z \in (0,\infty), \ x \neq z.
        \end{equation*}
        Hence, the $L^1$-norm of the second term in \eqref{8.1} can be estimated as in \eqref{7.1}, and for the third one we have
        \begin{align*}
            \left\| \int_{x/2}^{2x} \|f(z)\|_X \left\|\Psi_t(y+z) \right\|_{L^2(\G_+(x),\frac{dydt}{t^2})} dz \right\|_{L^1(0,\infty)}
                \leq & C \int_0^\infty \int_{x/2}^{2x} \frac{\|f(z)\|_X}{x+z} dz dx \\
                \leq & C  \int_0^\infty \|f(z)\|_X \int_{z/2}^{2z} \frac{dx}{x} dz
                \leq C \|f\|_{L^1((0,\infty),X)}.
        \end{align*}
        Thus, the proof of this lemma is completed.
    \end{proof}

    A straightforward adaptation of the arguments used in the proof of \cite[Theorem 4.8]{HW} allows us to show the
    following duality inequality.

    \begin{Lem}\label{Lem_4.8}
        Assume that $X$ is a Banach space and $0<q<\infty$. If $F : (0,\infty)^2 \longrightarrow X$ and
        $G: (0,\infty)^2 \longrightarrow X'$ are strongly measurable
        then, there exists $C>0$, such that
        $$ \int_0^\infty \int_0^\infty |\langle F(y,t) , G(y,t) \rangle| \frac{dy dt}{t}
            \leq C \int_0^\infty C_q^+(F)(y) A^+(G)(y)dy.$$
    \end{Lem}

    Assume that $\phi, \psi \in S_\lambda(0,\infty)$. We say that $\phi$ and $\psi$ are
    $\lambda$-complementary functions when
    $$\int_0^\infty h_\lambda(\phi)(y) h_\lambda(\psi)(y) \frac{dy}{y^{2\lambda+1}}=1.$$
    The Hankel transformation $h_\lambda$ is an automorphism in $S_\lambda(0,\infty)$ (\cite[Lemma 8]{Ze}).
    Suppose that $\phi \in S_\lambda(0,\infty)$ is not identically zero. Then, there exists an interval
    $I \subset (0,\infty)$ such that $h_\lambda(\phi)(y) \neq 0$, $y \in I$. We choose $\varphi \in C_c^\infty(0,\infty)$ such that
    $\varphi \geq 0$, $\varphi(x)=1$, $x \in I$, and we define
    $$\psi=\frac{h_\lambda(\varphi) \#_\lambda \bar{\phi}}{\displaystyle \int_0^\infty |h_\lambda(\phi)(x)|^2 \varphi(x) \frac{dx}{x^{3\lambda+1}}}.$$
    We have that $\psi \in S_\lambda(0,\infty)$ and by \eqref{hank_conv} the functions $\phi$ and $\psi$ are $\lambda$-complementary.
    Note that, since $\varphi \in C_c^\infty(0,\infty)$, $\int_0^\infty y^\lambda \psi(y)dy=0$. Indeed, by \cite[p. 104, (5.4.3)]{L}
    and \eqref{hank_conv} we have that
    \begin{align*}
        \int_0^\infty y^\lambda \psi(y)dy
            = & 2^{\lambda-1/2} \G(\lambda+1/2) \lim_{x \to 0^+} \int_0^\infty (xy)^{-\lambda+1/2} J_{\lambda-1/2}(xy) y^\lambda \psi(y)dy \\
            = & 2^{\lambda-1/2} \G(\lambda+1/2) \lim_{x \to 0^+} x^{-\lambda} h_\lambda( \psi)(x)\\
            = & \frac{2^{\lambda-1/2} \G(\lambda+1/2)}{\displaystyle \int_0^\infty |h_\lambda(\phi)(y)|^2 \varphi(y) \frac{dy}{y^{3\lambda+1}}}
                    \lim_{x \to 0^+} x^{-2\lambda} \varphi(x) h_\lambda(\bar{\phi})(x)
            =0.
    \end{align*}

    \begin{Lem}\label{Lem_polarizacion}
        Let $\lambda>1$ and $0<q<\infty$. Assume that $(1+x^2)^{-1}f \in L^1(0,\infty)$ and $g \in L_c^\infty(0,\infty)$.
        Moreover, suppose that $\phi, \psi \in S_\lambda(0,\infty)$ are a pair of
        $\lambda$-complementary functions such that $\int_0^\infty x^\lambda \phi(x)dx=\int_0^\infty x^\lambda \psi(x)dx=0$.
        If $C_q^+(f \#_\lambda \phi_{(t)}) \in L^\infty(0,\infty)$ and $A^+(g \#_\lambda \psi_{(t)}) \in L^1(0,\infty)$, then
        \begin{equation}\label{polarizacion}
            \int_0^\infty f(x) g(x) dx = \int_0^\infty \int_0^\infty (f \#_\lambda \phi_{(t)})(y) (g \#_\lambda \psi_{(t)})(y) \frac{dydt}{t}.
        \end{equation}
    \end{Lem}

    \begin{proof}
        Assume that $C_q^+(f \#_\lambda \phi_{(t)}) \in L^\infty(0,\infty)$ and $A^+(g \#_\lambda \psi_{(t)}) \in L^1(0,\infty)$.
        According to Lemma~\ref{Lem_4.8} the integral in the right hand side of \eqref{polarizacion} is absolutely convergent.
        Hence we can write
        \begin{equation}\label{R1}
            \int_0^\infty \int_0^\infty (f \#_\lambda \phi_{(t)})(y) (g \#_\lambda \psi_{(t)})(y) \frac{dydt}{t}
                = \lim_{N \to \infty} \int_{1/N}^N \int_0^\infty (f \#_\lambda \phi_{(t)})(y) (g \#_\lambda \psi_{(t)})(y) \frac{dydt}{t}.
        \end{equation}
        Our next objective is to establish that
        \begin{equation}\label{R2}
        \int_{1/N}^N \int_0^\infty (f \#_\lambda \phi_{(t)})(y) (g \#_\lambda \psi_{(t)})(y) dy \frac{dt}{t}
            = \int_0^\infty f(z)  \int_{1/N}^N \int_0^\infty \,_\lambda\tau_y (\phi_{(t)})(z) (g \#_\lambda \psi_{(t)})(y) dy \frac{dt}{t} dz, \ N \in \mathbb{N}.
        \end{equation}
        In order to do this we will prove that, for every $N \in \mathbb{N}$, there exists $C_N>0$ such that
        \begin{equation}\label{F2}
             \int_0^\infty \left|\,_\lambda\tau_y (\phi_{(t)})(z) (g \#_\lambda \psi_{(t)})(y) \right| dy
                \leq \frac{C_N}{1+z^2}, \quad  t \in (1/N,N) \text{ and } z \in (0,\infty).
        \end{equation}

        As in the proof of Lemma~\ref{Lem_Lp} we choose a function $\Phi \in S(\mathbb{R})$ such that
        $\Phi(z^2)=\phi(z)z^{-\lambda}$, $z \in (0,\infty)$. By using \cite[p. 86]{MS} we have that
        \begin{align}\label{B3}
            |\,_\lambda\tau_y (\phi_{(t)})(z)|
                \leq & C \frac{(yz)^\lambda}{t^{2\lambda+1}}
                        \int_0^\pi  (\sin \theta)^{2\lambda-1} \left|\Phi \left( \frac{(y-z)^2+2yz(1-\cos \theta)}{t^2} \right) \right| d\theta \nonumber \\
                \leq & C t (yz)^\lambda \int_0^\pi    \frac{(\sin \theta)^{2\lambda-1}}{\left[(y-z)^2+t^2+2yz(1-\cos \theta)\right]^{\lambda+1}}  d\theta
                \leq C \frac{t}{t^2+(y-z)^2}, \quad t,y,z \in (0,\infty).
        \end{align}

        Since $g \in L^\infty_c(0,\infty)$, by taking into account that
        $\frac{d}{dz}(z^{-\mu}J_\mu(z))=-z^{-\mu}J_{\mu+1}(z)$, $z \in (0,\infty)$, and that
        $\sqrt{z}J_\mu(z)$ is a bounded function on $(0,\infty)$, for every $\mu>-1/2$, we obtain
        $$\frac{d}{dx} h_\lambda(g)(x)
            = \frac{\lambda}{x} h_\lambda(g)(x) - h_{\lambda+1}(yg)(x), \quad x \in (0,\infty).$$
        By iterating this argument we conclude that $h_\lambda(g) \in C^\infty(0,\infty)$. Moreover,
        since $\sqrt{z}J_\mu(z)$ and $z^{-\mu}J_\mu(z)$ are bounded functions on $(0,\infty)$, for every $\mu>-1/2$,
        the functions $h_\lambda(g)$ and $\frac{d}{dx}h_\lambda(g)$ are bounded on $(0,\infty)$ because $\lambda>1$.
        According to
        \cite[Lemma 8]{Ze}, $h_\lambda(\psi_{(t)}) \in S_\lambda(0,\infty)$,
        $t>0$. Then, by taking into account \eqref{2.1} and that $\lambda>1$, we get
        \begin{align*}
            h_\lambda\left(u^{-\lambda} h_\lambda (\psi_{(t)})(u) h_\lambda(g)(u)\right)(x)
                = & \frac{1}{1+x^2} h_\lambda\left((1+\Delta_{\lambda,u})\left(u^{-\lambda} h_\lambda(\psi_{(t)})(u) h_\lambda(g)(u)\right)\right)(x),
                \quad x \in (0,\infty).
        \end{align*}
        On the other hand, $y^\lambda g \in L^1(0,\infty)$ and $\psi \in S_\lambda(0,\infty)$.
        Then, $y^\lambda(\psi_{(t)} \#_\lambda g) \in L^1(0,\infty)$ and $h_\lambda(\psi_{(t)})(y) h_\lambda(g)(y) \in L^1(0,\infty)$, and we obtain
       \begin{align*}
            (g\#_\lambda \psi_{(t)})(x)
                = & \frac{1}{1+x^2} h_\lambda \left((1+\Delta_{\lambda}) \left( u^{-\lambda} h_\lambda(\psi_{(t)})(u) h_\lambda(g)(u) \right) \right)(x), \quad x \in (0,\infty).
        \end{align*}
        We write
        $$1 + \Delta_\lambda
            = 1-(2\lambda+1)x^\lambda \left( \frac{1}{x} D \right) x^{-\lambda}
              - x^{\lambda+2} \left( \frac{1}{x} D \right)^2 x^{-\lambda}. $$
        Then,
        \begin{align*}
            (1+\Delta_{\lambda}) & \left( u^{-\lambda} h_\lambda(\psi_{(t)})(u) h_\lambda(g)(u) \right)
                =  (ut)^{-\lambda} h_\lambda(\psi)(ut) h_\lambda(g)(u) \\
                  &  - (2\lambda+1)u^\lambda \left[ \left( \frac{1}{u} D \right) \left( (ut)^{-\lambda} h_\lambda(\psi)(ut)  \right) u^{-\lambda}h_\lambda(g)(u)
                   +  (ut)^{-\lambda} h_\lambda(\psi)(ut)\left( \frac{1}{u} D \right) \left(u^{-\lambda} h_\lambda(g)(u) \right) \right] \\
                  & - u^{\lambda+2} \left[  \left( \frac{1}{u} D \right)^2 \left( (ut)^{-\lambda} h_\lambda(\psi)(ut) \right)
                   u^{-\lambda} h_\lambda(g)(u)
                   + 2 \left( \frac{1}{u} D \right) \left( (ut)^{-\lambda} h_\lambda(\psi)(ut) \right)  \right. \\
                  & \left. \cdot \left( \frac{1}{u} D \right)\left( u^{-\lambda} h_\lambda(g)(u) \right)
                   + (ut)^{-\lambda} h_\lambda(\psi)(ut) \left( \frac{1}{u} D \right)^2 \left( u^{-\lambda} h_\lambda(g)(u) \right) \right], \quad t,u \in (0,\infty).
        \end{align*}
        By using again that $h_\lambda(\psi) \in S_\lambda(0,\infty)$ and that the function $z^{-\mu} J_\mu(z)$
        is bounded on $(0,\infty)$, for every $\mu>-1/2$, we deduce that, for some $m \in \mathbb{N}$,
        $$\left| h_\lambda \left((1+\Delta_{\lambda}) \left( u^{-\lambda} h_\lambda(\psi_{(t)})(u) h_\lambda(g)(u) \right) \right)(x) \right|
            \leq C t^{-m}, \quad t, x \in (0,\infty).$$
        Putting together the above estimates we get,
        \begin{align*}
            \int_0^\infty \Big| \,_\lambda\tau_y (\phi_{(t)})(z) &(g \#_\lambda \psi_{(t)})(y) dy \Big|
                \leq \frac{C}{t^m} \int_0^\infty \frac{t}{(t^2+(y-z)^2)(1+y^2)} dy \\
                \leq & \frac{C}{t^m} \left( \frac{t}{t^2+z^2} \int_0^{z/2} \frac{dy}{1+y^2} + \frac{t}{1+z^2} \int_{z/2}^\infty \frac{dy}{t^2+(y-z)^2}  \right) \\
                \leq & \frac{CN^{m+1}}{1/N^2+z^2} \int_0^{\infty} \frac{dy}{1+y^2} + \frac{C N^{m+1}}{1+z^2} \int_{0}^\infty \frac{dy}{t^2+(y-z)^2}  \\
                \leq & \frac{C_N}{1+z^2}, \quad t \in (1/N,N), \ z \in (0,\infty).
        \end{align*}
        Thus \eqref{F2} is established.\\

        Since $g \in L^\infty_c(0,\infty)$, $g$ defines an element $T_g \in S_\lambda(0,\infty)'$ by
        $$\langle T_g , \varphi  \rangle
            = \int_0^\infty g(x) \varphi(x) dx, \quad \varphi \in S_\lambda(0,\infty).$$
        According to \cite{BeMa} we can write
        \begin{align*}
            \int_0^\infty \,_\lambda\tau_y (\phi_{(t)})(z)  (g \#_\lambda \psi_{(t)})(y) dy
                = & \left(\left( T_g  \#_\lambda \psi_{(t)} \right) \#_\lambda \phi_{(t)}\right)(z)
                =  \left( T_g  \#_\lambda \left(\psi_{(t)}  \#_\lambda \phi_{(t)}\right)\right)(z) \\
                = & \left( g  \#_\lambda \left(\psi_{(t)}  \#_\lambda \phi_{(t)}\right)\right)(z), \quad z \in (0,\infty).
        \end{align*}
        Hence, for every $\varphi \in S_\lambda(0,\infty)$, $T_g \#_\lambda \varphi$ is defined by
        $$(T_g \#_\lambda \varphi)(x)
            = \langle T_g , \,_\lambda\tau_x \varphi \rangle, \quad x \in (0,\infty),$$
        (see \cite{BeMa} for details about distributional Hankel convolution).\\

        Note that, by using the interchange formula, we get
        \begin{align*}
            h_\lambda \left( \psi_{(t)} \#_\lambda \phi_{(t)} \right)(x)
                = & x^{-\lambda} h_\lambda(\psi_{(t)})(x) h_\lambda(\phi_{(t)}) (x)
                = \frac{x^{-\lambda}}{t^\lambda} h_\lambda(\psi)(xt) \frac{1}{t^\lambda} h_\lambda(\phi) (xt) \\
                = &  \frac{1}{t^\lambda}  h_\lambda \left( \psi\#_\lambda \phi \right)(xt)
                =  h_\lambda \left( (\psi\#_\lambda \phi)_{(t)} \right)(x), \quad t,x \in (0,\infty).
        \end{align*}
        Hence, $\psi_{(t)} \#_\lambda \phi_{(t)} = (\psi\#_\lambda \phi)_{(t)} $, $t \in (0,\infty)$.
        Since $\psi\#_\lambda \phi \in S_\lambda(0,\infty)$ (\cite[Proposition 2.2, $(i)$]{MB})
        and $g \in L^\infty_c(0,\infty)$, from \eqref{B3} it follows that, for every $N \in \mathbb{N}$,
        \begin{align*}
            \int_{1/N}^N \int_0^\infty |g(y) \,_\lambda\tau_z (\psi_{(t)} \#_\lambda \phi_{(t)})(y)| \frac{dydt}{t}
                \leq & C \int_{1/N}^N \int_0^\infty  \frac{|g(y)| t}{t^2+(y-z)^2} \frac{dydt}{t}
                \leq  C \int_{1/N}^N \frac{dt}{t^2} \int_0^\infty  |g(y)| dy < \infty.
        \end{align*}
        In addition, we have that
        \begin{align*}
            \int_{1/N}^N \int_0^\pi (\sin \theta)^{2\lambda-1}
                    \frac{\left|(\psi \#_\lambda \phi)_{(t)}\left(\sqrt{(z-y)^2+2zy(1-\cos \theta)}\right) \right|}{((z-y)^2+2zy(1-\cos \theta))^{\lambda/2}} d\theta \frac{dt}{t}
                < \infty, \quad y,z \in (0,\infty), \ N \in \mathbb{N}.
        \end{align*}

        Thus, we obtain
        \begin{align}\label{R3}
            \int_{1/N}^N  \int_0^\infty \,_\lambda\tau_y (\phi_{(t)})(z) (g \#_\lambda \psi_{(t)})(y) dy  \frac{dt}{t}
                = &  \left( g  \#_\lambda \left( \int_{1/N}^N  \phi_{(t)} \#_\lambda \psi_{(t)} \frac{dt}{t}\right) \right)(z),
                \quad z \in (0,\infty).
        \end{align}
        The inner integral can be written as
        \begin{align*}
            \int_{1/N}^N  (\phi_{(t)} \#_\lambda \psi_{(t)})(u) \frac{dt}{t}
                = & \int_{1/N}^\infty  (\phi_{(t)} \#_\lambda \psi_{(t)})(u) \frac{dt}{t} - \int_N^\infty  (\phi_{(t)} \#_\lambda \psi_{(t)})(u) \frac{dt}{t} \\
                = & G_{(1/N)}(u)-G_{(N)}(u), \quad u \in (0,\infty) \text{ and } N \in \mathbb{N},
        \end{align*}
        where
        $$G(u)=\int_1^\infty (\phi_{(t)} \#_\lambda \psi_{(t)})(u) \frac{dt}{t}, \quad u \in (0,\infty).$$

        Since $\int_0^\infty y^\lambda \phi(y) dy = \int_0^\infty y^\lambda \psi(y) dy =0$, we deduce
        \begin{align*}
            \int_0^\infty  y^\lambda (\phi \#_\lambda \psi)(y)dy
                = & 2^{\lambda-1/2} \G(\lambda+1/2) \lim_{x \to 0^+} \int_0^\infty  y^\lambda (xy)^{-\lambda+1/2} J_{\lambda-1/2}(xy) (\phi \#_\lambda \psi)(y)  dy \\
                = & 2^{\lambda-1/2} \G(\lambda+1/2) \lim_{x \to 0^+} x^{-\lambda}  \int_0^\infty   (xy)^{1/2} J_{\lambda-1/2}(xy) (\phi \#_\lambda \psi)(y)  dy \\
                = & 2^{\lambda-1/2} \G(\lambda+1/2) \lim_{x \to 0^+} x^{-2\lambda} h_\lambda(\phi)(x) h_\lambda(\psi)(x)  \\
                = & \frac{1}{2^{\lambda-1/2}\G(\lambda+1/2)} \int_0^\infty y^\lambda \phi(y) dy \int_0^\infty y^\lambda \psi(y) dy=0.
        \end{align*}
        Moreover, by using \cite[(3), p. 135, and Lemma 5.4-1, (3)]{Ze2} and taking into account that
        $\phi \#_\lambda \psi \in S_\lambda(0,\infty)$, we obtain
        \begin{align}\label{null}
            \int_0^\infty  y^{2\lambda+2k} & \left( \frac{1}{y} \frac{d}{dy} \right)^k \left(y^{-\lambda}(\phi \#_\lambda \psi)(y) \right) dy \nonumber \\
                = & 2^{\lambda + k - 1/2} \G(\lambda+k+1/2) \lim_{x \to 0^+} x^{-\lambda-k}
                    h_{\lambda+k} \left( y^{\lambda+k} \left( \frac{1}{y} \frac{d}{dy} \right)^k \left(y^{-\lambda}(\phi \#_\lambda \psi)(y) \right) \right)(x)  \nonumber \\
                = & 2^{\lambda + k - 1/2} \G(\lambda+k+1/2) (-1)^k \lim_{x \to 0^+} x^{-\lambda}
                    h_{\lambda} \left(\phi \#_\lambda \psi\right)(x)  \nonumber \\
                = & \frac{2^{k} \G(\lambda+k+1/2) (-1)^k}{\G(\lambda+1/2)} \int_0^\infty  y^\lambda (\phi \#_\lambda \psi)(y)dy
                = 0 , \quad k \in \mathbb{N}.
        \end{align}
        Let $k,m \in \mathbb{N}$. Then, we can write
        \begin{align*}\label{exp}
            u^m \left( \frac{1}{u} \frac{d}{du} \right)^k \left( u^{-\lambda} G(u) \right)
                = & u^m \left( \frac{1}{u} \frac{d}{du} \right)^k \int_1^\infty \left( \frac{u}{t} \right)^{-\lambda} (\phi \#_\lambda \psi)\left(\frac{u}{t}\right) \frac{dt}{t^{2\lambda+2}} \nonumber \\
                = & u^m \int_1^\infty  \left( \frac{1}{v} \frac{d}{dv} \right)^k \left(v^{-\lambda} (\phi \#_\lambda \psi) (v) \right)_{|v=u/t} \frac{dt}{t^{2\lambda+2k+2}} \nonumber \\
                = &  u^{m-2k-2\lambda-1}  \int_0^u v^{2\lambda+2k} \left( \frac{1}{v} \frac{d}{dv} \right)^k \left(v^{-\lambda} (\phi \#_\lambda \psi)(v)\right) dv.
        \end{align*}
        By \eqref{null} and by applying L'Hôpital's rule we deduce that
        \begin{align*}
            \lim_{u \to 0^+} \left| u^m \left( \frac{1}{u} \frac{d}{du} \right)^k \left( u^{-\lambda} G(u) \right) \right|
            & \leq C \lim_{u \to 0^+} \left| u^{-2k-2\lambda-1}  \int_0^u v^{2\lambda+2k} \left( \frac{1}{v} \frac{d}{dv} \right)^k \left(v^{-\lambda} (\phi \#_\lambda \psi)(v)\right) dv \right| \\
            & \leq C \beta^\lambda_{0,k}(\phi \#_\lambda \psi) < \infty,
        \end{align*}
        and, if $m > 2k+2\lambda +1$,
        \begin{align*}
            \lim_{u \to \infty} \left| u^m \left( \frac{1}{u} \frac{d}{du} \right)^k \left( u^{-\lambda} G(u) \right) \right|
            & = \lim_{u \to \infty} \left| u^{m-2k-2\lambda-1}  \int_0^u v^{2\lambda+2k} \left( \frac{1}{v} \frac{d}{dv} \right)^k \left(v^{-\lambda} (\phi \#_\lambda \psi)(v)\right) dv \right| \\
            & \leq C \beta^\lambda_{m,k}(\phi \#_\lambda \psi) < \infty.
        \end{align*}
        When $m \leq  2k+2\lambda +1$, this last limit is equal to zero. We conclude that $G \in S_\lambda(0,\infty)$.\\

        Take $\alpha>0$ such that $\supp g \subset [0,\alpha]$. According to \eqref{B3}, we get
        \begin{align*}
            \left| g \#_\lambda G_{(s)}(z) \right|
                \leq & \|g\|_{L^\infty(0,\infty)} \int_0^\infty |\,_\lambda\tau_z (G_{(s)})(y)| dy
                \leq  C \int_0^\infty \frac{s}{(z-y)^2+s^2} dy \\
                \leq & C
                \leq \frac{C}{1+z^2}
                , \quad 0<z \leq 2\alpha \text{ and } s \in (0,\infty).
        \end{align*}
        Moreover, from the penultimate estimate in \eqref{B3}, it follows that
        \begin{align*}
            \left| g \#_\lambda G_{(s)}(z) \right|
                \leq & C \|g\|_{L^\infty(0,\infty)} \int_0^\alpha |\,_\lambda\tau_z (G_{(s)})(y)| dy
                \leq  C \int_0^\alpha \frac{(zy)^\lambda}{|z-y|^{2\lambda+1}} dy \\
                \leq &  \frac{C}{z^{1+\lambda}}
                \leq \frac{C}{1+z^2}
                , \quad z>2\alpha \text{ and } s \in (0,\infty),
        \end{align*}
        because $\lambda>1$.\\

        Hence, we conclude that
        \begin{equation}\label{F7}
            \sup_{N \in \mathbb{N}} \left|  \left( g  \#_\lambda \left( \int_{1/N}^N  \phi_{(t)} \#_\lambda \psi_{(t)} \frac{dt}{t}\right) \right)(z) \right|
                \leq \frac{C}{1+z^2}, \quad z \in (0, \infty).
        \end{equation}
        Since the function $\sqrt{z}J_\mu(z)$ is bounded in $(0,\infty)$, for every $\mu>-1/2$, \eqref{F7} and
        the interchange formula \eqref{hank_conv} imply that
        \begin{align*}
            h_\lambda \left( g  \#_\lambda \left( \int_{1/N}^N  \phi_{(t)} \#_\lambda \psi_{(t)} \frac{dt}{t}\right)  \right)(z)
                = & \int_{1/N}^N h_\lambda\left( g  \#_\lambda(\phi_{(t)} \#_\lambda \psi_{(t)})\right)(z) \frac{dt}{t} \\
                = & h_\lambda(g)(z) \int_{z/N}^{Nz}  h_\lambda(\phi)(y) h_\lambda(\psi)(y)\frac{dy}{y^{2\lambda+1}},
                \ z \in (0,\infty) \text{ and } N \in \mathbb{N}.
        \end{align*}
        Since $h_\lambda(\phi)$, $h_\lambda(\psi) \in S_\lambda(0,\infty)$ and
        $\int_0^\infty y^\lambda \phi(y)dy=\int_0^\infty y^\lambda \psi(y)dy=0$, \cite[Lemma 5.2-1]{Ze2}
        allows us to see that
        $\int_0^\infty |h_\lambda(\phi)(y)|  |h_\lambda(\psi)(y)| \frac{dy}{y^{2\lambda+1}}<\infty.$
        Moreover, $h_\lambda$ is bounded from $L^2(0,\infty)$ into itself.
        Then, by taking into account again that $\phi$ and $\psi$ are complementary functions, the dominated convergence
        theorem implies that
        $$\lim_{N \to \infty} h_\lambda(g) \int_{z/N}^{Nz}  h_\lambda(\phi)(y) h_\lambda(\psi)(y)\frac{dy}{y^{2\lambda+1}}
            =  h_\lambda(g), $$
        in $L^2(0,\infty)$, and also that
        $$\lim_{N \to \infty} \int_{1/N}^{N}  g  \#_\lambda(\phi_{(t)} \#_\lambda \psi_{(t)}) \frac{dt}{t}
            = g, $$
        in $L^2(0,\infty)$. There exists an increasing sequence $\{N_k\}_{k=1}^\infty \subset \mathbb{N}$ for which
        \begin{equation}\label{R4}
            \lim_{k \to \infty} \int_{1/N_k}^{N_k}  g  \#_\lambda(\phi_{(t)} \#_\lambda \psi_{(t)})(z) \frac{dt}{t}
            = g(z), \quad \text{a.e. } z \in (0,\infty).
        \end{equation}
        By applying again the dominated convergence theorem (see \eqref{F7}) and by \eqref{R1}, \eqref{R2}, \eqref{R3}, \eqref{R4} we obtain
        \begin{align*}
            \int_0^\infty \int_0^\infty  (f \#_\lambda \phi_{(t)})(y) (g \#_\lambda \psi_{(t)})(y) \frac{dydt}{t}
                =& \int_0^\infty f(z) \lim_{k \to \infty} \int_{1/N_k}^{N_k} \left( g  \#_\lambda \phi_{(t)} \#_\lambda \psi_{(t)}\right)(z)  \frac{dt}{t}dz \\
                = & \int_0^\infty f(z) g(z) dz.
        \end{align*}
        Thus, the proof is completed.
    \end{proof}

    \section{Proof of Theorem~\ref{Th_principal}}  \label{sec:proof_Th}

    In this section we present the proof of Theorem~\ref{Th_principal}.

    \subsection{Proof of Theorem~\ref{Th_principal}, $\pmb{(i)}$} \label{subsec:i}

    By taking into account Hölder's inequality it is sufficient to show $(i)$ when $q>1$.
    Hence, from now on we assume that $q>1$. Let $f \in BMO_o(\mathbb{R},X)$. Our objective is to prove
    that there exists $C>0$, that does not depend on $f$, such that,
    \begin{equation}\label{objt2}
        \left( \frac{1}{|I|} \int_I A^+(f \#_\lambda \phi_{(t)} \big||I|/2)^q(x) dx \right)^{1/q}
            \leq C \|f\|_{BMO_o(\mathbb{R},X)},
    \end{equation}
    for every bounded interval $I \subset (0,\infty)$.\\

    We take a bounded interval $I=(x_I-|I|/2,x_I+|I|/2) \subset (0,\infty)$. The function $f$ is decomposed as follows
    $$f \chi_{(0,\infty)} = (f-f_{3I})\chi_{3I} + (f-f_{3I})\chi_{(0,\infty) \setminus 3I} + f_{3I}=f_1 + f_2 + f_3,$$
    where $3I=(0,\infty) \cap (x_I-3|I|/2,x_I+3|I|/2)$. We write $F_i=f_i \#_\lambda \phi_{(t)}$, $i=1,2,3$.
    The estimation in \eqref{objt2} will be shown when we establish that
     \begin{equation}\label{27.1}
        \left( \frac{1}{|I|} \int_I A^+(F_i \big||I|/2)^q(x) dx \right)^{1/q}
            \leq C \|f\|_{BMO_o(\mathbb{R},X)}, \quad i=1,2,3.
    \end{equation}
    From \cite[Proposition 1.1]{NeWe} we deduce that
    $$\|F_1(y,t) \chi_{\G_+^{|I|/2}(x)}(y,t)\|_{\gamma\left(L^2((0,\infty)^2,\frac{dydt}{t^2});X\right)}
        \leq \|F_1(y,t) \chi_{\G_+(x)}(y,t)\|_{\gamma\left(L^2((0,\infty)^2,\frac{dydt}{t^2});X\right)}, \ x \in (0,\infty).$$
    Then, Lemma~\ref{Lem_Lp} implies that
    \begin{align*}
        \left( \frac{1}{|I|} \int_I A^+(F_1 \big||I|/2)^q(x) dx \right)^{1/q}
            \leq &  \frac{1}{|I|^q} \|A^+(F_1)\|_{L^q(0,\infty)}
            \leq C \left( \frac{1}{|3I|} \int_{3I} \|f(x)-f_{3I}\|_X^q dx \right)^{1/q} \\
            \leq & C \|f\|_{BMO_o(\mathbb{R},X)}.
    \end{align*}

    We now proof \eqref{27.1} for $i=2$. Observe that
    \begin{align*}
        A^+(F_2 \big||I|/2)(x)
            = & \| (f_2 \#_\lambda \phi_{(t)})(y) \chi_{\G_+^{|I|/2}(x)}(y,t)\|_{\gamma\left(L^2((0,\infty)^2,\frac{dydt}{t^2});X\right)} \\
            \leq & \int_{(0,\infty)\setminus 3I} \left\| (f(z)-f_{3I}) \,_\lambda\tau_y (\phi_{(t)})(z) \chi_{\G_+^{|I|/2}(x)}(y,t) \right\|_{\gamma\left(L^2((0,\infty)^2,\frac{dydt}{t^2});X\right)} dz \\
            = & \int_{(0,\infty)\setminus 3I} \left\| f(z)-f_{3I} \right\|_X \left\| \,_\lambda\tau_y (\phi_{(t)})(z) \right\|_{L^2\left(\G_+^{|I|/2}(x),\frac{dydt}{t^2}\right)} dz,
            \quad x \in I.
    \end{align*}
    Furthermore, by \eqref{B3},
    \begin{align*}
        \left\| \,_\lambda\tau_y (\phi_{(t)})(z) \right\|_{L^2\left(\G_+^{|I|/2}(x),\frac{dydt}{t^2}\right)}
            \leq & C \left( \int_{\G_+^{|I|/2}(x)} \frac{dydt}{|y-z|^4}  \right)^{1/2}
            \leq  C \frac{|I|}{|x_I-z|^2}, \quad x \in I, \ z \in (0,\infty)\setminus 3I.
    \end{align*}
    Hence, by \cite[Lemma 1.1, (a)]{G} we obtain
    \begin{align*}
        \Big( \frac{1}{|I|}  \int_I A^+(F_2 \big||I|/2)&^q(x) dx \Big)^{1/q}
            \leq C |I| \int_{(0,\infty)\setminus 3I} \frac{\left\| f(z)-f_{3I} \right\|_X}{|x_I-z|^2} dz \\
            \leq & C |I| \sum_{k=0}^\infty \int_{3^k |I| < |x_I-z| \leq 3^{k+1}|I|} \frac{\left\| f(z)-f_{3I} \right\|_X}{|x_I-z|^2} dz \\
            \leq & C \sum_{k=0}^\infty \frac{1}{3^k}
                \left( \frac{1}{3^k |I|}\int_{3^{k+1}I} \left\| f(z)-f_{3^{k+1}I} \right\|_X dz
                    + \left\| f_{3I}-f_{3^{k+1}I} \right\|_X \right)\\
            \leq & C \sum_{k=0}^\infty \frac{k}{3^k} \|f\|_{BMO_o(\mathbb{R},X)}
            \leq C \|f\|_{BMO_o(\mathbb{R},X)}.
    \end{align*}

    Next, we analyze \eqref{27.1} for $i=3$. In general, $f_{3I} \#_\lambda \phi_{(t)} \neq 0$, even when the function
    $x^\lambda \phi$ has vanishing integral. This is not the case when considering the usual convolution, see \cite[p. 48]{HW}.\\

    Observe that,
    \begin{align*}
        \frac{1}{|I|} \int_I A^+(F_3 \big||I|/2)^q(x) dx
            = & \|f_{3I}\|^q_X \frac{1}{|I|} \int_I \left( \int_{\G_+^{|I|/2}(x)} \left|\int_0^\infty \,_\lambda\tau_y (\phi_{(t)})(z) dz\right|^2 \frac{dydt}{t^2} \right)^{q/2} dx .
    \end{align*}

    In addition, taking into account that
    $$\|f_{3I}\|_X \leq C \frac{x_I+|I|}{|I|} \|f\|_{BMO_o(\mathbb{R},X)},$$
    and writing
    \begin{align*}
        \left|\int_0^\infty \,_\lambda\tau_y (\phi_{(t)})(z) dz\right|
            \leq & \int_0^{y/2} \left|\,_\lambda\tau_y (\phi_{(t)})(z) \right|dz
                + \left|\int_{y/2}^{2y} \,_\lambda\tau_y (\phi_{(t)})(z) dz\right|
               + \int_{2y}^\infty \left|\,_\lambda\tau_y (\phi_{(t)})(z) \right|dz \\
              = & \sum_{j=1}^3 J_j(y,t), \quad t,y \in (0,\infty),
    \end{align*}
    it is enough to prove that
    \begin{equation}\label{objt3}
         \frac{(x_I+|I|)^q}{|I|^{q+1}} \int_I \left( \int_0^{|I|/2} \int_{\max\{0,x-t\}}^{x+t} \left|J_j(y,t)\right|^2 \frac{dydt}{t^2} \right)^{q/2} dx
            \leq C, \quad j=1,2,3,
    \end{equation}
    being $C>0$ a constant independent of $I$ and $f$.\\

    By \eqref{B3} we can obtain the following estimations
    \begin{equation}\label{J1}
        J_1(y,t)
            \leq C \int_0^{y/2} \frac{(yz)^\lambda t}{[(y-z)^2+t^2]^{\lambda+1}} dz
            \leq C  \frac{y^{2\lambda+1} t}{(y^2+t^2)^{\lambda+1}}
            \leq C \frac{t}{t+y}, \quad t,y \in (0,\infty),
    \end{equation}
    \begin{equation}\label{J2}
        J_2(y,t)
            \leq C \int_{y/2}^{2y} \frac{(yz)^\lambda t}{[(y-z)^2+t^2]^{\lambda+1}} dy
            \leq C  \left( \frac{y}{t}\right)^{2\lambda+1}, \quad t,y \in (0,\infty),
    \end{equation}
    and
    \begin{align}\label{J3}
        J_3(y,t)
            \leq & C \int_{2y}^\infty \frac{(yz)^\lambda t}{[(y-z)^2+t^2]^{\lambda+1}} dz
            \leq  C y^\lambda t \int_{2y}^\infty \frac{z^\lambda }{(z^2+t^2)^{\lambda+1}} dz \nonumber \\
            \leq & C y^\lambda t \int_{2y}^\infty \frac{1 }{(z+t)^{\lambda+2}} dz
            \leq C \frac{y^\lambda t}{(y+t)^{\lambda+1}}
            \leq C \frac{t}{t+y}, \quad t,y \in (0,\infty).
    \end{align}
    Furthermore we will need the relation
    \begin{equation}\label{J2.2}
        J_2(y,t)
            \leq C  \frac{t}{y}, \quad t,y \in (0,\infty).
    \end{equation}
    To obtain this bound we have to proceed in a more involved way.
    We keep the same notation introduced in the proof of Lemma~\ref{Lem_Lp}. We have that
    \begin{align*}
        \left| \mathcal{K}_2(z;y,t) \right|
            & \leq C \frac{(yz)^\lambda}{t^{2\lambda+1}} \int_{\pi/2}^\pi (\sin \theta)^{2\lambda-1} \left( \frac{t^2}{y^2+z^2-2yz\cos \theta} \right)^{\lambda+1} d\theta
            \leq C \frac{t}{zy}, \quad t,y,z \in (0,\infty),
    \end{align*}
    and
    \begin{align*}
        \left| \mathcal{K}_{1,4}(z;y,t) \right|
            & \leq C \frac{(yz)^\lambda}{t^{2\lambda+1}} \int_{\pi/2}^\infty \theta^{2\lambda-1} \left( \frac{t^2}{yz\theta^2} \right)^{\lambda+1} d\theta
            \leq C \frac{t}{zy}, \quad t,y,z \in (0,\infty).
    \end{align*}
    Also by \cite[p. 483]{BCFR2} we obtain
    \begin{align*}
        \Big| \mathcal{K}_1(z;y,t) - \mathcal{K}_{1,1}(z;y,t) \Big|
            \leq &  C \frac{(yz)^\lambda}{t^{2\lambda+1}} \int_0^{\pi/2} \left| (\sin \theta)^{2\lambda-1} - \theta^{2\lambda-1} \right| \left( \frac{t^2}{t^2+(z-y)^2+2yz(1-\cos \theta)} \right)^{\lambda+1} d\theta \\
            \leq & t (yz)^\lambda \int_0^{\pi/2} \frac{\theta^{2\lambda+1}}{[t^2+(z-y)^2+zy\theta^2]^{\lambda+1}} d\theta \\
            \leq & C \frac{t}{zy} \left( 1 + \log_+ \frac{z}{|z-y|} \right), \quad t \in (0,\infty), \ y/2<z<2y,
    \end{align*}
    and
    \begin{align*}
        \Big| \mathcal{K}_{1,1}(z;y,t) - \mathcal{K}_{1,2}(z;y,t) \Big|
            \leq &  C \frac{(yz)^\lambda}{t^{2\lambda+1}} \int_0^{\pi/2}  \theta^{2\lambda-1}
                \left| \Phi\left( \frac{(y-z)^2+2yz(1-\cos \theta)}{t^2} \right) - \Phi\left( \frac{(y-z)^2+yz\theta^2}{t^2} \right)\right| d\theta \\
            \leq &  C \frac{(yz)^\lambda}{t^{2\lambda+1}} \int_0^{\pi/2}  \theta^{2\lambda+3} \frac{yz}{t^2}
                \left( \frac{t^2}{t^2+(y-z)^2+yz\theta^2} \right)^{\lambda+2}  d\theta \\
            \leq & t (yz)^\lambda \int_0^{\pi/2} \frac{\theta^{2\lambda+1}}{[t^2+(z-y)^2+zy\theta^2]^{\lambda+1}} d\theta \\
            \leq & C \frac{t}{zy} \left( 1 + \log_+ \frac{z}{|z-y|} \right), \quad t \in (0,\infty), \ y/2<z<2y.
    \end{align*}
    With all the above estimations and \eqref{nula} we get
    \begin{align*}
        |J_2(y,t)|
            \leq & \Big| \int_{y/2}^{2y} \Big(\mathcal{K}_2(z;y,t) + \left( \mathcal{K}_1(z;y,t) - \mathcal{K}_{1,1}(z;y,t) \right)
                    + \left( \mathcal{K}_{1,1}(z;y,t) - \mathcal{K}_{1,2}(z;y,t) \right) \\
                 &  + \Psi_t(y-z) - \mathcal{K}_{1,4}(z;y,t) \Big) dz \Big| \\
            \leq & C \Big( \frac{t}{y}  \int_{y/2}^{2y} \frac{1}{z}\left( 1 + \log_+ \frac{z}{|z-y|} \right) dz
                    + \left| \int_\mathbb{R} \Psi_t(y-z) dz \right|  + \int_{(-\infty,y/2) \cup (2y,\infty)} |\Psi_t(y-z)| dz \Big)    \\
            \leq & C \frac{t}{y}, \quad t,y \in (0,\infty).
    \end{align*}

    We are ready to prove \eqref{objt3}. It is important to analyze carefully the region of integration, that is the truncated cone
    $\G_+^{|I|/2}(x)$, $x \in (0,\infty)$. We distinguish two cases. Assume first $2|I|<x_I$. Then, using estimates \eqref{J1}, \eqref{J3}
    and \eqref{J2.2} we get, for $j=1,2,3$,
    \begin{align*}
        \frac{(x_I+|I|)^q}{|I|^{q+1}}  &\int_I \left( \int_0^{|I|/2} \int_{\max\{0,x-t\}}^{x+t} \left|J_j(y,t)\right|^2 \frac{dydt}{t^2} \right)^{q/2} dx
           \leq C \frac{x_I^q}{|I|^{q+1}} \int_I \left( \int_0^{|I|/2} \int_{x_I-|I|}^{x_I+|I|} \frac{t^2}{y^2} \frac{dydt}{t^2} \right)^{q/2} dx \\
           \leq & C \frac{x_I^q}{|I|^{q/2}} \left( \frac{1}{x_I-|I|} - \frac{1}{x_I+|I|}\right)^{q/2}
           \leq C \frac{x_I^q}{|I|^{q/2}} \frac{|I|^{q/2}}{(x_I-|I|)^q}
           \leq C.
    \end{align*}

    Suppose now that $|I|/2 \leq x_I \leq 2|I|$. If $x \in I$, then $x_I-|I|/2<x<|I|/2$ or $|I|/2 \leq x< x_I+|I|/2$. We are going to consider each
    situation separately. In the sequel $\int_a^b g(z)dz=0$ provided that $a \geq b$.
    By \eqref{J1} and \eqref{J3} we can write, for $j=1,3$,
    \begin{align*}
        \frac{(x_I+|I|)^q}{|I|^{q+1}}  &\int_{x_I-|I|/2}^{|I|/2} \left( \int_0^{|I|/2} \int_{\max\{0,x-t\}}^{x+t} \left|J_j(y,t)\right|^2 \frac{dydt}{t^2} \right)^{q/2} dx \\
            \leq & \frac{C}{|I|}  \int_{0}^{|I|/2} \left( \left\{ \int_0^x \int_{x-t}^{x+t} + \int_x^{|I|/2} \int_0^{x+t} \right\} \frac{dydt}{(t+y)^2} \right)^{q/2} dx \\
            \leq & C \left(\frac{1}{|I|}  \int_{0}^{|I|/2} \left(  \int_0^x \int_0^{2x} \frac{dydt}{x^2} \right)^{q/2} dx
                    + \frac{1}{|I|}  \int_{0}^{|I|/2} \left( \int_x^{|I|/2} \int_0^{2t}  \frac{dydt}{t^2} \right)^{q/2} dx \right) \\
            \leq & C \left(1 + \frac{1}{|I|}  \int_{0}^{|I|/2} \left( \log \frac{|I|}{2x} \right)^{q/2} dx \right)
            \leq C \left(1 + \frac{1}{|I|}  \int_{0}^{|I|/2} \left( \frac{|I|}{2x} \right)^{1/2} dx \right)
            \leq C,
    \end{align*}
    because $\log z \leq z^\alpha$, for every $z>0$ and $\alpha>0$.
    If $j=2$ we apply \eqref{J2} and \eqref{J2.2} to obtain
    \begin{align*}
        \frac{(x_I+|I|)^q}{|I|^{q+1}}  &\int_{x_I-|I|/2}^{|I|/2} \left( \int_0^{|I|/2} \int_{\max\{0,x-t\}}^{x+t} \left|J_2(y,t)\right|^2 \frac{dydt}{t^2} \right)^{q/2} dx \\
            \leq &  \frac{C}{|I|}  \int_{0}^{|I|/2} \left( \left\{\int_0^{x/2} \int_{x-t}^{x+t} + \int_{x/2}^x \int_{x-t}^{x+t}+ \int_x^{|I|/2} \int_0^{x+t} \right\} \left|J_2(y,t)\right|^2 \frac{dydt}{t^2} \right)^{q/2} dx \\
            \leq & C \left( \frac{1}{|I|}  \int_{0}^{|I|/2} \left( \int_0^{x/2} \int_{x-t}^{x+t}  \frac{dydt}{y^2} \right)^{q/2} dx
                        +   \frac{1}{|I|}  \int_{0}^{|I|/2} \left( \int_{x/2}^x \int_0^{2x} \left( \frac{y}{t} \right)^{4\lambda+2} \frac{dydt}{t^2} \right)^{q/2} dx \right. \\
                 & \left.      +   \frac{1}{|I|}  \int_{0}^{|I|/2} \left(  \int_x^{|I|/2} \int_0^{2t} \left( \frac{y}{t} \right)^{4\lambda+2} \frac{dydt}{t^2} \right)^{q/2} dx \right)\\
            \leq & C \left( \frac{1}{|I|}  \int_{0}^{|I|/2} \left( \int_0^{x/2} \int_{0}^{2x}  \frac{dydt}{x^2} \right)^{q/2} dx
                        +   \frac{1}{|I|}  \int_{0}^{|I|/2} \left( \int_{x/2}^x  \frac{x^{4\lambda+3}}{t^{4\lambda+4}} dt \right)^{q/2} dx \right. \\
                 & \left.      +   \frac{1}{|I|}  \int_{0}^{|I|/2} \left(  \int_x^{|I|/2} \int_0^{2t} \frac{dydt}{t^2} \right)^{q/2} dx \right)
            \leq  C.
    \end{align*}

    On the other hand, applying again \eqref{J1} and \eqref{J3}, for $j=1,3$, we get
    \begin{align*}
        \frac{(x_I+|I|)^q}{|I|^{q+1}}  &\int_{|I|/2}^{x_I+|I|/2} \left( \int_0^{|I|/2} \int_{\max\{0,x-t\}}^{x+t} \left|J_j(y,t)\right|^2 \frac{dydt}{t^2} \right)^{q/2} dx \\
            \leq & \frac{C}{|I|}  \int_{|I|/2}^{3|I|} \left( \int_0^{|I|/2} \int_{x-t}^{x+t} \frac{dydt}{(t+y)^2} \right)^{q/2} dx
            \leq  \frac{C}{|I|}  \int_{|I|/2}^{3|I|} \left( \int_0^{|I|/2} \int_{0}^{4|I|} \frac{dydt}{|I|^2} \right)^{q/2} dx
            \leq C,
    \end{align*}
    and by \eqref{J2} and \eqref{J2.2}, it follows that
    \begin{align*}
        \frac{(x_I+|I|)^q}{|I|^{q+1}}  & \int_{|I|/2}^{x_I+|I|/2} \left( \int_0^{|I|/2} \int_{\max\{0,x-t\}}^{x+t} \left|J_2(y,t)\right|^2 \frac{dydt}{t^2} \right)^{q/2} dx \\
            \leq & C \left( \frac{1}{|I|}   \int_{|I|/2}^{3|I|} \left( \int_0^{|I|/4} \int_{x-t}^{x+t} \frac{dydt}{y^2} \right)^{q/2} dx
                  +  \frac{1}{|I|}   \int_{|I|/2}^{3|I|} \left( \int_{|I|/4}^{|I|/2} \int_{x-t}^{x+t} \left( \frac{y}{t} \right)^{4\lambda+2} \frac{dydt}{t^2} \right)^{q/2} dx \right) \\
            \leq & C \left( \frac{1}{|I|}   \int_{|I|/2}^{3|I|} \left( \int_0^{|I|/4} \int_{0}^{4|I|} \frac{dydt}{|I|^2} \right)^{q/2} dx
                  +  \frac{1}{|I|}   \int_{|I|/2}^{3|I|} \left( \int_{|I|/4}^{|I|/2} \int_{0}^{4|I|}  \frac{y^{4\lambda+2}}{|I|^{4\lambda+4}}  dydt \right)^{q/2} dx \right) \\
            \leq & C.
    \end{align*}
    Note that all the constants $C$ that appear do not depend on $I$ and $f$.
    This shows \eqref{objt3} and therefore the proof of \eqref{objt2} is finished.
    \begin{flushright}
        \qed
    \end{flushright}

    \subsection{Proof of Theorem~\ref{Th_principal}, $\pmb{(ii)}$} \label{subsec:ii}

    Suppose that $C_q^+(f \#_\lambda \phi_{(t)}) \in L^\infty(0,\infty)$, for some $0<q<\infty$. Since $X$
    is a UMD space, it is also reflexive and therefore  $BMO_o(\mathbb{R},X)$ is the dual space of
    $H_o^1(\mathbb{R},X^*)$ (see \cite{Bl} and \cite[p. 466]{BCFR2}). If $L^\infty_{c,o}(\mathbb{R})$ denotes
    the space of odd bounded functions with compact support in $\mathbb{R}$, then $L^\infty_{c,o}(\mathbb{R})$
    is a dense subspace of $H^1_{o}(\mathbb{R})$. Hence, according to \cite[Lemma 2.4]{Hy},
    $L^\infty_{c,o}(\mathbb{R}) \otimes X^*$ is a dense subspace of $H^1_{o}(\mathbb{R},X^*)$. Our objective is to
    see that $f \in BMO_o(\mathbb{R},X)$. In order to prove this it is sufficient to show that, for a certain
    $C>0$,
    $$\left| \int_\mathbb{R} \langle f(x),g(x) \rangle dx \right|
        \leq C \|g\|_{H^1_{o}(\mathbb{R},X^*)}, \quad g \in L^\infty_{c,o}(\mathbb{R}) \otimes X^*.$$
    Let $g \in L^\infty_{c,o}(\mathbb{R}) \otimes X^*$.
    We can write $g=\sum_{i=1}^N a_i g_i$, where $a_i \in X^*$ and $g_i \in L^\infty_{c,o}(\mathbb{R})$,
    $i=1, \dots, N$. It is clear that
    $$\int_\mathbb{R} \langle f(x),g(x) \rangle dx
        = 2 \int_0^\infty \langle f(x),g(x) \rangle dx
        = 2 \sum_{i=1}^N \int_0^\infty f_i(x)g_i(x) dx,$$
    where $f_i(x)=\langle f(x),a_i \rangle$, $x \in \mathbb{R}$. Let $i=1, \dots, N$.
    We have that $(1+x^2)^{-1}f_i \in L^1(0,\infty)$.
    Since $\mathbb{C}$ is a UMD space, by using Lemma~\ref{Lem_L1}, $A^+(g_i \#_\lambda \psi_{(t)}) \in L^1(0,\infty)$,
    for every $\psi \in S_\lambda(0,\infty)$ such that $\int_0^\infty x^\lambda \psi(x)dx=0$.
    Moreover,
    $$\|C_q^+(f_i \#_\lambda \phi_{(t)})\|_{L^\infty(0,\infty)}
        \leq \|a_i\|_{X^*} \|C_q^+(f \#_\lambda \phi_{(t)})\|_{L^\infty(0,\infty)}
        < \infty.$$
    We choose a function $\psi \in S_\lambda(0,\infty)$ that is complementary to $\phi$
    and such that $\int_0^\infty x^\lambda \psi(x)dx=0$.
    By applying now Lemma~\ref{Lem_polarizacion} we get
    \begin{align*}
        \int_\mathbb{R} \langle f(x),g(x) \rangle dx
            = & 2\sum_{i=1}^N \int_0^\infty \int_0^\infty (f_i \#_\lambda \phi_{(t)})(y) (g_i \#_\lambda \psi_{(t)})(y) \frac{dydt}{t} \\
            = & 2\sum_{i=1}^N \int_0^\infty \int_0^\infty  \left\langle (f \#_\lambda \phi_{(t)})(y) , a_i \right\rangle    (g_i \#_\lambda \psi_{(t)})(y) \frac{dydt}{t} \\
            = & 2 \int_0^\infty \int_0^\infty  \left\langle (f \#_\lambda \phi_{(t)})(y) , \sum_{i=1}^N a_i (g_i \#_\lambda \psi_{(t)})(y) \right\rangle     \frac{dydt}{t} \\
            = & 2\int_0^\infty \int_0^\infty \langle (f \#_\lambda \phi_{(t)})(y) , (g \#_\lambda \psi_{(t)})(y) \rangle \frac{dydt}{t}.
    \end{align*}
    According to Lemma~\ref{Lem_L1} and Lemma~\ref{Lem_4.8}, it follows that
    \begin{align*}
        \left| \int_\mathbb{R} \langle f(x),g(x) \rangle dx \right|
            \leq & C \int_0^\infty \int_0^\infty \left| \langle (f \#_\lambda \phi_{(t)})(y) , (g \#_\lambda \psi_{(t)})(y) \rangle \right| \frac{dydt}{t} \\
            \leq & C \int_0^\infty C_q^+(f \#_\lambda \phi_{(t)})(x) A^+(g \#_\lambda \psi_{(t)})(x)  dx \\
            \leq & C \|C_q^+(f \#_\lambda \phi_{(t)})\|_{L^\infty(0,\infty)} \| A^+(g \#_\lambda \psi_{(t)}) \|_{L^1(0,\infty)} \\
            \leq & C \|C_q^+(f \#_\lambda \phi_{(t)})\|_{L^\infty(0,\infty)} \| g \|_{H_o^1(\mathbb{R},X^*)}.
    \end{align*}
    Hence, we conclude that $f \in BMO_o(\mathbb{R},X)$ and
    $$\|f\|_{BMO_o(\mathbb{R},X)}
        \leq C \|C_q^+(f \#_\lambda \phi_{(t)})\|_{L^\infty(0,\infty)}.$$
    \begin{flushright}
        \qed
    \end{flushright}

    \section{Proof of Theorem~\ref{Th_caract}}  \label{sec:Poisson}

    \subsection{Proof of $\pmb{(i) \Rightarrow (ii)}$}
        Assume that $X$ is a UMD Banach space and $f$ is an odd $X$-valued function such that $(1+x^2)^{-1}f \in L^1(\mathbb{R},X)$.
        According to Theorem~\ref{Th_Poisson} we need only to show that, for a certain $C>0$,
        \begin{equation*}
            \left\|C_q^+\left(\int_0^\infty f(z) t D_{\lambda,z} P_t^\lambda(z,y) dz\right)\right\|_{L^\infty(0,\infty)}
                \leq C \|f\|_{BMO_o(\mathbb{R},X)}.
        \end{equation*}
        We are going to follow the same  procedure as in the proof of $(i)$ in Theorem~\ref{Th_principal}.
        We fix  a bounded interval
        $I=(x_I-|I|/2,x_I+|I|/2) \subset (0,\infty)$ and we decompose $f$ as usual by
        $$f \chi_{(0,\infty)}
            = (f-f_{3I})\chi_{3I} + (f-f_{3I})\chi_{(0,\infty) \setminus 3I} + f_{3I}
            =f_1 + f_2 + f_3,$$
        where $3I=(0,\infty) \cap (x_I-3|I|/2,x_I+3|I|/2)$. Our objective is to see that
        \begin{equation}\label{objt4}
        \left( \frac{1}{|I|} \int_I A^+(G_i \big||I|/2)^q(x) dx \right)^{1/q}
            \leq C \|f\|_{BMO_o(\mathbb{R},X)},
        \end{equation}
        where $C>0$ does not depend on $I$ and $f$, and
        $$G_i(y,t)= \int_0^\infty f_i(z) t D_{\lambda,z} P_t^\lambda(z,y) dz, \quad t,y \in (0,\infty), \quad i=1,2,3.$$

        Lemma~\ref{Lem_Lp_dx} implies \eqref{objt4} for $i=1$. To show \eqref{objt4} for $i=2$
        it is enough to establish that
        \begin{equation}\label{E4}
            \left| t D_{\lambda,z} P_t^\lambda(z,y) \right|
                \leq C \frac{t}{t^2 + (y-z)^2}, \quad t,y,z \in (0,\infty).
        \end{equation}
        We can write
        \begin{align*}
            \left| t D_{\lambda,z} P_t^\lambda(z,y) \right|
                \leq C \left( P_t^\lambda(z,y) +  t (yz)^\lambda \int_0^{\pi}
                    \frac{ (\sin \theta)^{2\lambda-1}y(1-\cos \theta)}{[(y-z)^2+t^2+2yz(1-\cos \theta)]^{\lambda+3/2}} d\theta  \right),
                \ t,y,z \in (0,\infty).
        \end{align*}
        Observe that
        \begin{align}\label{31}
           t (yz)^\lambda & \int_{\pi/2}^\pi\frac{ (\sin \theta)^{2\lambda-1}y(1-\cos \theta)}{[(y-z)^2+t^2+2yz(1-\cos \theta)]^{\lambda+3/2}} d\theta
            \leq  C \frac{t(yz)^\lambda}{t^2+(y-z)^2}  \int_{\pi/2}^\pi\frac{ (\sin \theta)^{2\lambda-1}}{[y^2+z^2+t^2-2yz\cos \theta]^{\lambda}} d\theta \nonumber\\
            \leq & C \frac{t}{t^2+(y-z)^2}  \left( \frac{yz}{y^2+z^2}  \right)^\lambda
            \leq  C \frac{t}{t^2+(y-z)^2}, \quad t,y,z \in (0,\infty).
        \end{align}
        In order to analyze the integral over $(0,\pi/2)$ we distinguish two situations.
        Suppose first that $t,y,z \in (0,\infty)$ and $z \geq |y-z|+t$. Then
        \begin{align}\label{32}
            t (yz)^\lambda & \int_0^{\pi/2} \frac{ (\sin \theta)^{2\lambda-1}y(1-\cos \theta)}{[(y-z)^2+t^2+2yz(1-\cos \theta)]^{\lambda+3/2}} d\theta
                \leq C t y (yz)^\lambda \int_0^{\pi/2} \frac{ \theta^{2\lambda+1}}{[(y-z)^2+t^2+yz\theta^2]^{\lambda+3/2}} d\theta \nonumber \\
                \leq & C \frac{t}{z[(y-z)^2+t^2]^{1/2}}  \int_0^{\infty} \frac{ u^{2\lambda+1}}{(1+u^2)^{\lambda+3/2}} du
                \leq  C \frac{t}{(y-z)^2+t^2}.
        \end{align}
        On the other hand, if $t,y,z \in (0,\infty)$ and $z < |y-z|+t$, we obtain
        \begin{align}\label{33}
            t (yz)^\lambda & \int_0^{\pi/2} \frac{ (\sin \theta)^{2\lambda-1}y(1-\cos \theta)}{[(y-z)^2+t^2+2yz(1-\cos \theta)]^{\lambda+3/2}} d\theta
            \leq C \frac{ t y }{(|y-z|^2+t^2)^{3/2}}\int_0^{\pi/2} \theta^{2\lambda+1}\left(\frac{yz}{yz\theta^2}\right)^{\lambda}  d\theta \nonumber \\
            \leq & C \frac{ t (|y-z|+z) }{(|y-z|^2+t^2)^{3/2}}
            \leq C \frac{ t  }{(y-z)^2+t^2}.
        \end{align}
        By \cite[p. 86]{MS}, \eqref{31}, \eqref{32} and \eqref{33} we conclude that \eqref{E4} holds.\\

        Finally, we are going to show \eqref{objt4} for $i=3$. We have that
        \begin{align*}
            \left| \int_0^\infty t D_{\lambda,z} P_t^\lambda(z,y) dz \right|
                \leq &  \int_0^{y/2} \left|t D_{\lambda,z} P_t^\lambda(z,y)  \right|dz
                   + \left| \int_{y/2}^{2y} t D_{\lambda,z} P_t^\lambda(z,y) dz \right|
                   +  \int_{2y}^\infty \left|t D_{\lambda,z} P_t^\lambda(z,y)\right| dz  \\
                = & \sum_{j=1}^3 \mathcal{M}_j(y,t), \quad t,y \in (0,\infty).
        \end{align*}
        Our objective is to establish for $\mathcal{M}_j$, $j=1,2,3$, estimates similar
        to \eqref{J1}, \eqref{J2}, \eqref{J3} and \eqref{J2.2}. From \eqref{E4} we deduce that
        $$\mathcal{M}_j(y,t) \leq C \frac{t}{t+y}, \quad t,y \in (0,\infty), \ j=1,3.$$
        We also have that
        \begin{align*}
            \mathcal{M}_2(y,t)
                \leq C \int_{y/2}^{2y} \frac{t(zy)^\lambda y}{t^{2\lambda+3}} dz
                \leq C \left( \frac{y}{t} \right)^{2\lambda+2}, \quad t,y \in (0,\infty).
        \end{align*}
        Note that the exponent in the last inequality differs from the one appearing in \eqref{J2},
        but the computations made in the proof $(i)$ in
        Theorem~\ref{Th_principal} can be done in the same way.\\

        Our last step will be to justify that
        \begin{equation*}\label{ls}
            \mathcal{M}_2(y,t)
                \leq C \frac{t}{y}, \quad t,y \in (0,\infty).
        \end{equation*}
        By keeping the notation in the proof of Lemma~\ref{Lem_Lp_dx} we can write
        \begin{align*}
            \mathcal{M}_2(y,t)
            \leq & \Big| \int_{y/2}^{2y} \Big( \mathcal{L}_2(z;y,t) + \left( \mathcal{L}_1(z;y,t) - \mathcal{L}_{1,1}(z;y,t) \right)
                    + \left( \mathcal{L}_{1,1}(z;y,t) - \mathcal{L}_{1,2}(z;y,t) \right) \\
                 &  + \left( \mathcal{L}_{1,6}(z;y,t) - \mathcal{L}_{1,4}(z;y,t) \right)
                    +  t\partial_z (P_t(y-z)) \Big) dz \Big|.
        \end{align*}
        We treat each summand. We get
        \begin{align*}
            |\mathcal{L}_2(z;y,t)|
                \leq & C t^2 (yz)^\lambda  \int_{\pi/2}^\pi \frac{ (\sin \theta)^{2\lambda-1}(|z-y|+y)}{[(y-z)^2+t^2+2yz(1-\cos \theta)]^{\lambda+2}} d\theta \\
                \leq & C t (yz)^\lambda  \int_{\pi/2}^\pi \frac{ (\sin \theta)^{2\lambda-1}}{[y^2+z^2+t^2-2yz\cos \theta]^{\lambda+1}} d\theta \\
                \leq & C \frac{t}{yz} \left( \frac{yz}{y^2+z^2} \right)^{\lambda+1}
                \leq  C \frac{t}{yz}, \quad t,y,z \in (0,\infty).
        \end{align*}
        Also, by \cite[p. 483]{BCFR2} it follows that
        \begin{align*}
            |\mathcal{L}_1(z;y,t) - \mathcal{L}_{1,1}(z;y,t)|
                \leq & C t^2 (yz)^\lambda  \int_{0}^{\pi/2} \frac{ \theta^{2\lambda+1}(|z-y|+y\theta^2)}{[(y-z)^2+t^2+yz\theta^2]^{\lambda+2}} d\theta \\
                \leq & C t (yz)^\lambda  \int_{0}^{\pi/2} \frac{ \theta^{2\lambda+1}}{[(y-z)^2+t^2+yz\theta^2]^{\lambda+1}} d\theta \\
                \leq & C \frac{t}{yz} \left( 1 + \log_+ \frac{z}{|z-y|} \right), \quad t \in (0,\infty), \ y/2<z<2y,
        \end{align*}
        and
        \begin{align*}
            |\mathcal{L}_{1,1}(z;y,t) - \mathcal{L}_{1,2}(z;y,t)|
                \leq & C t^2 (yz)^\lambda  \int_{0}^{\pi/2} \frac{ \theta^{2\lambda+3}y}{[(y-z)^2+t^2+yz\theta^2]^{\lambda+2}} d\theta \\
                \leq & C \frac{t}{yz} \left( 1 + \log_+ \frac{z}{|z-y|} \right), \quad t \in (0,\infty), \ y/2<z<2y.
        \end{align*}
        Moreover, by \eqref{GB} we obtain
        \begin{align*}
            |\mathcal{L}_{1,6}(z;y,t) - \mathcal{L}_{1,4}(z;y,t)|
                \leq & C \left( \mathcal{I}_\lambda(z;y,t) + t(yz)^\lambda  \int_{\pi/2}^\infty \frac{ \theta^{2\lambda-1}}{[(y-z)^2+t^2+yz\theta^2]^{\lambda+1}} d\theta  \right),
                \ t,y,z \in (0,\infty),
        \end{align*}
        where $\mathcal{I}_\lambda$ is defined in \eqref{GC}. We have that
        \begin{align*}
            t(yz)^\lambda  \int_{\pi/2}^\infty \frac{ \theta^{2\lambda-1}}{[(y-z)^2+t^2+yz\theta^2]^{\lambda+1}} d\theta
                \leq C t(yz)^\lambda  \int_{\pi/2}^\infty \frac{ \theta^{2\lambda-1}}{(yz\theta^2)^{\lambda+1}} d\theta
                \leq C \frac{t}{yz}, \ t,y,z \in (0,\infty).
        \end{align*}
        Also, we write
        \begin{align*}
            \mathcal{I}_\lambda(z;y,t)
                = & t^2(yz)^\lambda  \int_0^\infty \frac{ \theta^{2\lambda+1}y}{[(y-z)^2+t^2+yz\theta^2]^{\lambda+2}} d\theta
                  - t^2(yz)^\lambda  \int_{\pi/2}^\infty \frac{ \theta^{2\lambda+1}y}{[(y-z)^2+t^2+yz\theta^2]^{\lambda+2}} d\theta \\
                = & \mathcal{I}^1_\lambda(z;y,t) - \mathcal{I}^2_\lambda(z;y,t), \quad t,y,z \in (0,\infty).
        \end{align*}
        The change of variables $u=\theta \sqrt{zy/((z-y)^2+t^2)}$ leads to
        \begin{align*}
            \mathcal{I}^1_\lambda(z;y,t)
                \leq & C t^2(yz)^{\lambda-1/2}  \int_0^\infty \frac{ \theta^{2\lambda-1}}{[(y-z)^2+t^2+yz\theta^2]^{\lambda+1}} d\theta
                =C  \frac{t^2}{\sqrt{zy}} \frac{1}{(z-y)^2+t^2} \int_0^\infty \frac{u^{2\lambda-1}}{(1+u^2)^{\lambda+1}} dy\\
                \leq & C \frac{t}{y} P_t(y-z), \quad t \in (0,\infty), \ y/2<z<2y.
        \end{align*}
        We have that
        \begin{align*}
            \mathcal{I}^2_\lambda(z;y,t)
                \leq & C t(yz)^{\lambda}  \int_{\pi/2}^\infty \frac{ \theta^{2\lambda+1}y}{(yz\theta^2)^{\lambda+3/2}} d\theta
                \leq  C \frac{t}{yz}, \quad t \in (0,\infty), \ y/2<z<2y.
        \end{align*}
        Putting together the above estimates and taking into account that
        $\int_\mathbb{R} \partial_z P_t(u-z)du=0$, $z \in \mathbb{R}$, we obtain
        \begin{align*}
            \mathcal{M}_2(y,t)
                \leq & C \left[ \frac{t}{y} \int_{y/2}^{2y} \left( \frac{1}{z}\left( 1 + \log_+ \frac{z}{|z-y|} \right) + P_t(y-z) \right)  dz
                                      +      \int_{(-\infty,y/2) \cup (2y,\infty)} P_t(y-z) dz \right] \\
                \leq & C \frac{t}{y} \left[  \int_{1/2}^{2} \frac{1}{u}\left( 1 + \log_+ \frac{1}{|1-u|} \right)du + \int_{\mathbb{R}}P_t(u)du + 1 \right] \\
                \leq & C \frac{t}{y}, \quad t,y \in (0,\infty).
        \end{align*}
        Then, \eqref{objt4} for $i=3$ can be proved by proceeding as in the proof of the corresponding
        property in Theorem~\ref{Th_principal}.\\

        Thus the proof of this part of Theorem~\ref{Th_caract} is completed.
    \begin{flushright}
        \qed
    \end{flushright}

    \subsection{Proof of $\pmb{(ii) \Rightarrow (i)}$}
        According to \cite[Theorem 2.1]{BFMT}, in order to see that $X$ is a UMD Banach space,
        it is sufficient to show that the Riesz transform $R_\lambda$
        associated with the Bessel operator $\Delta_\lambda$, can be extended to $L^p((0,\infty),X)$ as a bounded
        operator from $L^p((0,\infty),X)$ into itself, for some $1<p<\infty$. We recall that
        $R_\lambda f =D_\lambda \Delta_\lambda^{-1/2} f $, for every $f \in C^\infty_c(0,\infty)$, where
        $D_\lambda=x^\lambda \frac{d}{dx} x^{-\lambda}$ and
        $$\Delta^{-1/2}_\lambda f(x)
            = \int_0^\infty P_t^\lambda(f)(x)dt, \quad f \in C^\infty_c(0,\infty) \text{ and } x \in (0,\infty).$$

        $R_\lambda$ is a Calderón-Zygmund operator (\cite[Theorem 4.2 and Corollary 4.4]{BBFMT})
        and it can be extended to $L^p(0,\infty)$, $1 \leq p < \infty$, as the principal value integral operator
        $$R_\lambda f(x)
            = \lim_{\varepsilon \to 0^+} \int_{0, |x-y|>\varepsilon}^\infty R_\lambda(x,y) f(y)dy,
            \quad \text{a.e. } x \in (0,\infty),$$
        for every $f \in L^p(0,\infty)$, $1 \leq p < \infty$, where
        $$R_\lambda(x,y)
            = x^\lambda \int_0^\infty \partial_x (x^{-\lambda }P_t^\lambda(x,y)) dt, \quad x,y \in (0,\infty), \ x \neq y.$$
        $R_\lambda$ is a bounded operator from $L^p(0,\infty)$ into itself, for every $1<p<\infty$, and
        from $L^1(0,\infty)$ into $L^{1,\infty}(0,\infty)$. We denote by $R_\lambda^*$ the adjoint operator of
        $R_\lambda$ in $L^2(0,\infty)$. $R_\lambda^*$ is a principal value integral operator given by, for every $g \in L^2(0,\infty)$,
        $$R_\lambda^* g(x)
            = \lim_{\varepsilon \to 0^+} \int_{0, |x-y|>\varepsilon}^\infty R_\lambda(y,x) g(y)dy,
            \quad \text{a.e } x \in (0,\infty).$$
        We can see that
        $$R_\lambda^* g = D_\lambda^* \Delta_{\lambda+1}^{-1/2} g, \quad g \in C^\infty_c(0,\infty).$$
        Since $\Delta_{\lambda+1}=D_\lambda D_\lambda^*$, $R_\lambda^*$ is the Riesz transformation associated with the Bessel operator $\Delta_{\lambda+1}$. Moreover, $R^*_\lambda$ is a Calderón-Zygmund operator (see \cite[Proposition 4.1]{BBFMT}).\\

        $R_\lambda^*$ is extended to $L^\infty_c(0,\infty) \otimes X$ in the natural way. Our objective is to show that
        \begin{equation}\label{H1}
            \|(R_\lambda^*f)_o\|_{BMO_o(\mathbb{R},X)}
                \leq C \|f_o\|_{BMO_o(\mathbb{R},X)}, \quad f \in L^\infty_c(0,\infty) \otimes X,
        \end{equation}
        where $(R_\lambda^*f)_o$ and $f_o$ denote the odd extensions to $\mathbb{R}$ of $R_\lambda^*f$ and $f$,
        respectively. \eqref{H1} implies that $X$ is a UMD space. Indeed, from \eqref{H1} we deduce that
        \begin{equation}\label{H2}
            \|R_\lambda^*f\|_{BMO((0,\infty),X)}
                \leq C \|f\|_{L^\infty((0,\infty),X)}, \quad f \in L^\infty_c(0,\infty) \otimes X.
        \end{equation}
        Suppose that $E$ is a finite dimensional subspace of $X$. Then,
        $L^\infty_c(0,\infty) \otimes E = L^\infty_c((0,\infty),E)$. By \eqref{H2} we have that
        \begin{equation}\label{H3}
            \|R_\lambda^*f\|_{BMO((0,\infty),E)}
                \leq C \|f\|_{L^\infty((0,\infty),E)}, \quad f \in L^\infty_c((0,\infty),E).
        \end{equation}
        Note that the constant $C$ in \eqref{H3} does not depend on the subspace $E$. Since $R_\lambda^*$ is a Calderón-Zygmund operator,
        \cite[Theorem 4.1]{MTX} allows us to obtain that
        \begin{equation*}
            \|R_\lambda^*f\|_{L^2((0,\infty),E)}
                \leq C \|f\|_{L^2((0,\infty),E)}, \quad f \in L^\infty_c((0,\infty),E),
        \end{equation*}
        where $C>0$ is independent of $E$.\\

        Hence, we conclude that
        \begin{equation*}
            \|R_\lambda^*f\|_{L^2((0,\infty),X)}
                \leq C \|f\|_{L^2((0,\infty),X)}, \quad f \in L^\infty_c(0,\infty) \otimes X,
        \end{equation*}
        and, since \cite[Theorem 2.1]{BBFMT} also works when $R_\lambda^*$ replaces $R_\lambda$, $X$ is UMD.\\

        We are going to prove \eqref{H1}. Let $f \in L^\infty_c(0,\infty) \otimes X$.
        The odd extension $f_o$ of $f$ to $\mathbb{R}$ can be written $f=\sum_{j=1}^N b_j f_j$,
        where $f_j \in L^\infty_{c,o}(\mathbb{R})$ and $b_j \in X$, $j=1, \dots, N \in \mathbb{N}$.
        We have that
        $$R_\lambda^*(f)= \sum_{j=1}^N b_j R_\lambda^*(f_j),$$
        on $(0,\infty)$.\\

        Since $R^*_\lambda$ is a bounded operator from $L^2(0,\infty)$ into itself (\cite[Theorem 4.2]{BBFMT}),
        Hölder's inequality leads to
        \begin{align*}
            \int_\mathbb{R} \frac{\|R_\lambda^*(f)(x)\|_X}{1+x^2} dx
            \leq & C \sum_{j=1}^N \|b_j\|_X \|f_j\|_{L^2(0,\infty)}.
        \end{align*}
       Hence, we can apply \eqref{tdt} to obtain
        \begin{equation}\label{H7}
            \|(R_\lambda^*(f))_o\|_{BMO_o(\mathbb{R},X)}
                \leq C \|C_q^+(t \partial_t P_t^\lambda(R_\lambda^*(f))) \|_{L^\infty(0,\infty)}.
        \end{equation}

        We define the function $Q_t^\lambda(f)$ by
        $$Q_t^\lambda(f)(x)
            = \int_0^\infty Q_t^\lambda(x,y) f(y)dy
            = \sum_{j=1}^N b_j \int_0^\infty Q_t^\lambda(x,y) f_j(y)dy, \quad t,x \in (0,\infty),$$
        where
        $$Q_t^\lambda(x,y)
            = \frac{2\lambda}{\pi} (xy)^\lambda
            \int_0^\pi \frac{(x-y\cos \theta)(\sin \theta)^{2\lambda-1}}{(x^2+y^2+t^2-2xy\cos \theta)^{\lambda+1}}d\theta,
            \quad t,x,y \in (0,\infty).$$
        The function $Q_t^\lambda(f)$ is called $\Delta_\lambda$-conjugate to the Poisson integral $P_t^\lambda(f)$,
        because of the following Cauchy-Riemann type equations
        $$ D_{\lambda} P_t^\lambda(f) = \partial_t Q_t^\lambda(f),
            \quad
           D^*_{\lambda} Q_t^\lambda(f) = \partial_t P_t^\lambda(f),$$
        being $D^*_\lambda=-x^{-\lambda} \frac{d}{dx} x^{\lambda}$.\\

        We define the function
        $\mathbb{Q}_t^\lambda(g)$ by
        $$\mathbb{Q}_t^\lambda(g)(x)
            = \int_0^\infty Q_t^\lambda(y,x) g(y)dy, \quad t,x \in (0,\infty),$$
        where $g \in L^2(0,\infty)$. The following Cauchy--Riemann type equations are satisfied
        $$ D^*_{\lambda,x} P_t^{\lambda+1}(g) = \partial_t \mathbb{Q}_t^\lambda(g),
           \quad
           D_{\lambda,x} \mathbb{Q}_t^\lambda(g) = \partial_t P_t^{\lambda+1}(g).$$
        Moreover, by using the Hankel transformation (see \cite[(16.5)]{MS})
        we can prove that
        $$P_t^\lambda(R_\lambda^* g)
            = \mathbb{Q}_t^\lambda(g), \quad g \in L^2(0,\infty).$$
        The above relations allows us to write
        $$\partial_t P_t^\lambda(R_\lambda^* f)(x)
            = \partial_t \mathbb{Q}_t^\lambda (f)(x)
            = D^*_{\lambda,x} P_t^{\lambda+1}(f)(x), \quad t,x \in (0,\infty).$$
        Also, from \cite[(5.3.5)]{L} it follows that
        $$D^*_{\lambda,x} P_t^{\lambda+1}(f)(x)
            = \int_0^\infty D_{\lambda,z}P_t^\lambda(x,z)f(z)dz, \quad t,x \in (0,\infty). $$
        Hence, from \eqref{tdx} we deduce that
        \begin{equation}\label{H8}
            \left\|C_q^+(t \partial_t P_t^\lambda(R_\lambda^*(f)))\right\|_{L^\infty(0,\infty)}
            \leq C \|f_o\|_{BMO_o(\mathbb{R},X)}.
        \end{equation}
        By combining \eqref{H7} and \eqref{H8} we get \eqref{H1} and the proof of this part of Theorem~\ref{Th_caract}
        is finished.
    \begin{flushright}
        \qed
    \end{flushright}


\end{document}